\documentclass[12pt,a4paper,reqno]{amsart}
\usepackage[utf8]{inputenc}
\usepackage[T1]{fontenc}
\usepackage{indentfirst}
\usepackage[english]{babel}
\usepackage{enumitem,physics}
\usepackage{amsmath,amsfonts,amssymb,amsopn,amscd,amsthm}
\usepackage{mathrsfs,shuffle,bbm}
\usepackage{stmaryrd,csquotes}
\usepackage{graphicx}
\usepackage[style=alphabetic,sorting=nyt,maxnames=4]{biblatex}
\usepackage[dvipsnames]{xcolor}
\usepackage[colorlinks=true,citecolor=DarkOrchid,linkcolor=NavyBlue]{hyperref}
\usepackage{pgfplots}
\usepackage{tikz}
\usetikzlibrary{patterns,arrows.meta}
\pgfplotsset{compat=1.17}
\usepackage[a4paper,twoside,margin=0.8in]{geometry}
    \newtheorem{definition}{Definition}
    \newtheorem{lemma}[definition]{Lemma}
    \newtheorem{theorem}[definition]{Theorem}
    \newtheorem{proposition}[definition]{Proposition}
    \newtheorem{corollary}[definition]{Corollary}
    
    \newtheorem{main}{Theorem}
    
    \theoremstyle{remark}
    \newtheorem{example}[definition]{Example}
    \newtheorem{remark}[definition]{Remark}

    \makeatletter
    \def\thm@space@setup{\thm@preskip=0.5cm   \thm@postskip=0.5cm}
    \makeatother

\newcommand{\N}{\mathbb{N}}
\newcommand{\Z}{\mathbb{Z}}

\newcommand{\R}{\mathbb{R}}
\newcommand{\C}{\mathbb{C}}
\newcommand{\Tor}{\mathbb{T}}
\newcommand{\For}{\mathbb{F}}
\newcommand{\E}{\mathrm{e}}
\newcommand{\I}{\mathrm{i}}
\newcommand{\Sym}{\mathrm{Sym}}
\newcommand{\esper}{\mathbb{E}}

\newcommand{\proba}{\mathbb{P}}

\newcommand{\card}{\mathrm{card}}
\newcommand{\eps}{\varepsilon}

\newcommand{\lle}{\left[\!\left[} 
\newcommand{\rre}{\right]\!\right]}
\newcommand{\sym}{\mathfrak{S}}
\newcommand{\DD}[1]{\,d\hspace{-0.3mm}{#1}}

\newcommand{\comment}[1]{}
\renewcommand{\Re}{\mathrm{Re}}
\renewcommand{\Im}{\mathrm{Im}}
\newcommand{\dtv}{d_{\mathrm{TV}}}
\newcommand{\frakp}{\mathfrak{p}}
\newcommand{\frake}{\mathfrak{e}}
\newcommand{\frakh}{\mathfrak{h}}
\newcommand{\frakP}{\mathfrak{P}}
\newcommand{\frakE}{\mathfrak{E}}
\newcommand{\frakH}{\mathfrak{H}}
\newcommand{\Irr}{\mathrm{Irr}}

\setlist[enumerate]{itemsep=10pt,topsep=10pt}
\setlist[itemize]{itemsep=5pt,topsep=5pt}

\title[Mod-poisson approximation schemes and higher-order Chen--Stein inequalities]{Mod-poisson approximation schemes and\\ higher-order Chen--Stein inequalities}

\author{Pierre-Lo\"ic M\'eliot}
\address{Institut de math\'ematiques d'Orsay, Universit\'e Paris-Saclay, France}
\email{pierre-loic.meliot@universite-paris-saclay.fr}
\thanks{}

\author{Ashkan Nikeghbali}
\address{Institute of Mathematics, Universit\"at Z\"urich, Switzerland}
\email{ashkan.nikeghbali@math.uzh.ch}
\thanks{}

\author{Gabriele Visentin}
\address{Department of Mathematics, RiskLab, ETH Zürich, Switzerland}
\email{gabriele.visentin@math.ethz.ch}
\thanks{}

\begin{document}
\begin{abstract} In this article, we provide an extension of the Chen--Stein inequality for Poisson approximation in the total variation distance for sums of independent Bernoulli random variables in two ways. We prove that:
\begin{itemize}
\item we can improve the rate of convergence (hence the quality of the approximation) by using explicitly constructed signed or positive probability measures;
\item we can extend the setting to possibly dependent random variables.
\end{itemize}
The framework which allows this is that of mod-Poisson convergence and more precisely those mod-Poisson convergent sequences whose residue functions can be expressed as a specialization of the generating series of elementary symmetric functions. This combinatorial reformulation allows us to have a general and unified framework in which we can fit the classical setting of sums of independent Bernoulli random variables as well as other examples coming e.g. from probabilistic number theory and random permutations.
\end{abstract}

\maketitle
\tableofcontents
\bigskip

\section{Higher-order Chen--Stein inequalities}
\subsection{Poisson approximation and the Chen--Stein inequality} This article is concerned with the problem of approximation of the distribution of positive integer-valued random variables $X_{n\geq 1}$ stemming from number theory or from combinatorics. The simplest case is when $X_n = \sum_{i=1}^n Y_i$ is a sum of independent Bernoulli random variables with $p_i=\proba[Y_i=1]=1-\proba[Y_i=0] $; the parameters $p_i$ belong to $[0,1]$ and are arbitrary. Le Cam's inequality \cite{LeCam60}, which follows immediately from the subadditivity of the total variation distance with respect to convolution of measures, ensures that if $\lambda_n = \sum_{i=1}^n p_i$ and $(\sigma_n)^2 = \sum_{i=1}^n (p_i)^2$, then
$$\dtv(\mu_{X_n},\mathrm{Poisson}(\lambda_n)) = \frac{1}{2}\sum_{k=0}^\infty \left|\mu_{X_n}(k) - \frac{\E^{-\lambda_n}\,(\lambda_n)^k}{k!}\right| \leq (\sigma_n)^2
$$
where $\mu_{X_n}(k) = \proba[X_n=k]$. A more precise inequality due to Chen ensures that one can divide the right-hand side by $\lambda_n$ while keeping the inequality (see \cite{Chen74,BE83,BHS92,Steele94}):
$$\dtv(\mu_{X_n},\mathrm{Po}(\lambda_n)) \leq \frac{1-\E^{-\lambda_n}}{\lambda_n} \,(\sigma_n)^2 \leq \frac{(\sigma_n)^2}{\lambda_n}.$$
This follows from the adaptation of Stein's method to the Poisson distribution $\mathrm{Po}(\lambda)$, which satisfies the functional equation $\esper[X\,f(X)]=\lambda\,\esper[f(X+1)]$. In \cite[Theorem 4.2]{EK09}, a correction term for the Poisson distribution $\mathrm{Po}(\lambda_n)$ has been computed, allowing to have an upper bound of order $(\sigma_n)^3/(\lambda_n)^{3/2}$ for the total variation distance between $\mu_{X_n}$ and the modified Poisson distribution. This correction procedure has been perfected in \cite{CDMN20}: by using Fourier analysis arguments instead of the Chen--Stein method, a signed measure $\nu_n^{(r)}=\nu_n^{(r)}(\{p_1,\ldots,p_n\})$ has been constructed for any $r \geq 1$, with the property that if $(\sigma_n)^2= \sum_{i=1}^n (p_i)^2$ stays bounded, then
\begin{equation}
\dtv\left(\mu_{X_n},\nu_n^{(r)}\right) = O\left((\lambda_n)^{-\frac{r+1}{2}}\right).\label{eq:previous_1}
\end{equation}
However, the constant in the $O(\cdot)$ in the right-hand side of Equation \eqref{eq:previous_1} grows with $r$ and with $(\sigma_n)^2$, and by looking at the details of the proof of \cite[Theorem 3.11]{CDMN20}, one can only give a constant which grows like $\E^{(\sigma_n)^2}$, which might be much larger than $(\lambda_n)^{\frac{r+1}{2}}$. The purpose of this article is to give a better \emph{unconditional bound}, which is also much \emph{more general}, because it holds for approximations of discrete distributions which are not convolutions of Bernoulli distributions. In this setting, the set of probabilities $\{p_1,\ldots,p_n\}$ will be replaced by a square-summable sequence $A$ which encodes the properties of the random model.
\medskip

\subsection{Erd\texorpdfstring{\H{o}}{ö}s--Kac central limit theorem and its Poisson refinement}\label{sub:erdos_kac}
Given an integer $k \geq 1$, we denote $\omega(k)$ its number of distinct prime divisors, with the multiplicities not taken into account; for instance, $\omega(120)=\omega(2^3\times 3\times 5)=3$. The random variable $\omega_{n}$ is then defined by
$$\omega_n = \omega(U_{\lle 1,n\rre}),\quad \text{with $U_{\lle 1,n\rre}$ uniformly distributed on }\lle 1,n\rre = \{1,2,3,\ldots,n\}.$$
The asymptotics of the distribution of $\omega_n$ is one of the first result from probabilistic number theory. Indeed Erd\H{o}s and Kac proved in \cite{EK40} that the following convergence in law holds:
$$\frac{\omega_n-\log \log n}{\sqrt{\log \log n}} \rightharpoonup_{n \to \infty} \mathcal{N}(0,1).$$
The Gaussian approximation is far from being accurate and Rényi and Turán proved that the error in Kolmogorov distance is of order $O((\log \log n)^{-\frac{1}{2}})$, see \cite{RT58}. A better discrete approximation is provided by the Poisson law with parameter $\log \log n$. Indeed, the Fourier transform of $\omega_n$ can be estimated by the Selberg--Delange method:
\[
\esper[\E^{\I \xi \omega_n}] = \E^{(\log \log n + \gamma)(\E^{\I \xi}-1)}\,\left(\psi_\omega(\xi)+O\!\left(\frac{1}{\log n}\right)\right) \tag{MP$\omega$}\label{eq:selberg_delange}
\]
with $\gamma=0.577\ldots$ equal to the Euler--Mascheroni constant, and 
$$\psi_\omega(\xi) = \prod_{n \in \N^*} \left(1+\frac{\E^{\I \xi}-1}{n}\right)\E^{-\frac{\E^{\I \xi}-1}{n}}\,\, \prod_{p \in \mathbb{P}}\left(1+\frac{\E^{\I \xi}-1}{p}\right)\E^{-\frac{\E^{\I \xi}-1}{p}}.$$
We recognize the Fourier transform of the distribution $\mathrm{Po}(\log \log n + \gamma)$, multiplied by a residue which converges uniformly towards a smooth function $\psi(\xi)$ on the torus $\Tor = \R/2\pi \Z$. This can be used in order to prove that
$$\dtv(\omega_n,\mathrm{Po}(\log \log n+\gamma)) = O((\log \log n)^{-1}).$$
In \cite[Theorem 4.9]{CDMN20}, some signed measures $\nu_n^{(r\geq 1)}$ have been constructed with the property that 
\begin{equation}
\dtv(\omega_n,\nu_n^{(r)}) \,(\log \log n)^{\frac{r+1}{2}} \to_{n \to \infty} C_r\label{eq:previous_2}
\end{equation}
for some explicit constants $C_{r\geq 1}$; the case $r=1$ corresponds to the Poisson distribution with parameter $\log \log n + \gamma$. The speed of convergence in Equation \eqref{eq:previous_2} is a $O((\log \log n)^{-\frac{1}{2}})$, so the estimate above is not as good as an unconditional upper bound, which we aim to obtain at the end of this article.
\medskip

\subsection{Mod-Poisson approximation schemes}
The two asymptotic estimates \eqref{eq:previous_1} and \eqref{eq:previous_2} are very similar, but the first one regards the approximation in distribution of a sum of independent random variables, whereas the later estimate is about $\omega_n$ whose law does not admit such a representation. The connection between these two models comes from the asymptotics of their Fourier transforms. Consider a sum $X_n = \sum_{i=1}^n Y_i$ of independent Bernoulli random variables, with each $Y_i$ of parameter $p_i$, and with a sequence of parameters $\rho=(p_i)_{i \geq 1}$ such that $\sum_{i=1}^\infty p_i = +\infty$ and $\sum_{i=1}^\infty (p_i)^2 < +\infty$. Since $\esper[\E^{\I \xi Y_i}] = (1+p_i(\E^{\I \xi}-1))$, we can write:
\[
\esper[\E^{\I \xi X_n}] = \E^{\lambda_n (\E^{\I \xi}-1)}\,\,\left(\psi_{\rho}(\xi)+O\!\left(\sum_{i>n} (p_i)^2 \right)\right) \tag{MP$\rho$}\label{eq:mod_poisson_basic}
\]
with $\lambda_n = \sum_{i=1}^n p_i$, and 
$$\psi_\rho(\xi) = \prod_{i=1}^\infty\left(1+p_i(\E^{\I \xi}-1)\right)\E^{-p_i(\E^{\I \xi}-1)}.$$
In Equations \eqref{eq:selberg_delange} and \eqref{eq:mod_poisson_basic}, we have:
\begin{itemize}
    \item a sequence of random variables $(X_n)_{n \in \N}$ with values in $\N$,
    \item and a sequence of parameters $(\lambda_n)_{n \in \N}$ growing to infinity,
\end{itemize}  
such that the ratio of Fourier transforms 
\[
\frac{\esper[\E^{\I \xi X_n}]}{\esper[\E^{\I \xi \mathrm{Po}(\lambda_n)}]} = \esper[\E^{\I \xi X_n}]\,\E^{-\lambda_n(\E^{\I \xi}-1)} = \psi_n(\xi)\tag{MP}\label{eq:mod_poisson}
\]
converges on the unit circle towards an analytic function $\psi(\xi)$. Notice that the residue $\psi(\xi)$ has the same form for the two models: a convergent infinite product of terms $(1+p(\E^{\I \xi}-1))\,\E^{-p(\E^{\I \xi}-1)}$. This is a general phenomenon for discrete random models stemming from number theory or combinatorics, and it will enable us to use a unified approach with the same techniques. A sequence of integer-valued random variables $(X_n)_{n \in \N}$ for which the ratios of Fourier transforms $\psi_n$ defined in Formula \eqref{eq:mod_poisson} converge uniformly on the unit circle towards a continuous function $\psi$ is called \emph{mod-Poisson convergent} with parameters $(\lambda_n)_{n \in \N}$. This notion has been introduced in \cite{BKN09,KN10} and studied thoroughly in \cite{DKN15,FMN16,CDMN20}; in these later articles, the exponent $\phi(\xi)=\E^{\I \xi}-1$ is sometimes replaced by the exponent of a general infinitely divisible distribution supported by the lattice $\Z$ (\emph{mod-$\phi$ convergence}). Here, we shall focus on the mod-Poisson case in order to prove precise estimates of the distribution of $X_n$ when the residue defined by Formula \eqref{eq:mod_poisson} converges.

\begin{definition}[Approximation scheme of order $r$]
Let $(X_n)_{n \in \N}$ be a sequence of random variables with values in $\N$, which converges mod-Poisson with parameters $(\lambda_n)_{n \in \N}$ and limiting residue 
$$\psi(\xi) = \lim_{n \to \infty} \psi_n(\xi) = \lim_{n \to \infty} \esper[\E^{\I \xi X_n}]\,\E^{-\lambda_n(\E^{\I \xi} -1)}.$$
 We assume that the residues $\psi_n$ and $\psi$ are given by the following convergent power series:
$$\psi_n(\xi) = 1+\sum_{s=1}^\infty b_{s,n}\,(\E^{\I \xi}-1)^s\qquad;\qquad \psi(\xi)=1+\sum_{s=1}^\infty b_s\,(\E^{\I \xi}-1)^s.$$
Then, the \emph{approximation scheme of order $r$} of the law $\mu_n$ of $X_n$ is the signed measure $\nu_n^{(r)}$ with Fourier transform
$$\widehat{\nu}_n^{(r)}(\xi) = \sum_{k=0}^\infty \widehat{\nu}_n^{(r)}(k)\,\E^{\I k \xi} = \E^{\lambda_n(\E^{\I \xi}-1)}\,\left(1+\sum_{s=1}^r b_{s,n}\,(\E^{\I \xi}-1)^s\right).$$
The \emph{derived approximation scheme of order $r$} of $X_n$ is defined similarly, by truncation of the coefficients of the limiting residue $\psi(\xi)$:
$$\widehat{\nu}^{(r)}_{n,*}(\xi) = \E^{\lambda_n(\E^{\I \xi}-1)}\,\left(1+\sum_{s=1}^r b_{s}\,(\E^{\I \xi}-1)^s\right).$$
\end{definition}

Notice that by definition, $\nu_n^{(0)}=\nu_{n,*}^{(0)}$ is the Poisson law $\mathrm{Po}(\lambda_n)$. For $r \geq 1$,  an explicit formula for $\nu_n^{(r)}$ is provided by \cite[Lemma 3.8]{CDMN20}:
$$\nu^{(r)}_n(k) = \sum_{0\leq t \leq s \leq r} (-1)^{s-t}\,\binom{s}{t}\,b_{s,n}\, \nu_n^{(0)}(k-t),$$
and similarly for the derived approximation scheme. This explicit formula implies that $\nu_n^{(r)}$ is a well defined signed measure on $\N$, and that $\sum_{k \in \N} \nu_n^{(r)}(k)=1$. Our goal is then to control
$$\dtv(\mu_n,\nu_n^{(r)}) = \frac{1}{2}\sum_{k =0}^\infty \left|\proba[X_n=k] - \nu_n^{(r)}(k)\right|$$
and $\dtv(\mu_n,\nu_{n,*}^{(r)})$. Before going on, let us remark that up to a modification of the sequence $(\lambda_n)_{n \in \N}$, we can assume $b_{1,n}=b_1=0$. Indeed, replacing $\lambda_n$ by $\lambda_n+b_{1,n}$ removes the term $b_{1,n}$ from the power series $\psi_n(\xi)$ (this also modifies the other coefficients $b_{s\geq 2,n}$).\medskip

Informally, our main results are the following:
\begin{enumerate}
    \item If $(X_n)_{n \in \N}$ converges in the mod-Poisson sense and if 
    $$|b_{s,n}|\leq \left(\frac{\tau_n}{\sqrt{s}}\right)^s$$
    for some constants $\tau_n>0$, then $\dtv(\mu_n,\nu_n^{(r)}) \leq C \left(\frac{D\,\tau_n}{\sqrt{\lambda_n}}\right)^{r+1}$ for some universal constants $C$ and $D$ (see Theorem \ref{main:A} for a more precise statement).
    \item In the particular case where $X_n = \sum_{i=1}^n Y_i$ is the sum of independent Bernoulli variables, the condition above is satisfied with $(\tau_n)^2$ proportional to $(\sigma_n)^2=\sum_{i=1}^n (p_i)^2$. Therefore, we get a higher-order Chen--Stein inequality (see Theorem \ref{main:B}).
    \item If the mod-Poisson convergence admits a Bernoulli-like asymptotic residue $\psi$ (with the same form as in Equations \eqref{eq:selberg_delange} and \eqref{eq:mod_poisson_basic}), and if the convergence $\psi_n\to \psi$ \begin{itemize}
        \item can be expanded to a complex disc containing the unit circle,
        \item is fast enough,
    \end{itemize}
    then we get an unconditional upper bound  (see Theorem \ref{main:C})
    $$\dtv(\mu_n,\nu_{n,*}^{(r)})\leq \frac{C(\psi)}{(\lambda_n)^{\frac{r+1}{2}}}.$$ 
\end{enumerate}
\medskip

\subsection{Outline of the paper} The theory of mod-convergent sequences with respect to an arbitrary reference infinitely divisible distribution $\phi$ has been developed in \cite{DKN15,FMN16,FMN19,BMN19,CDMN20}. In the specific case where the reference law is the Poisson distribution, the residue of mod-Poisson convergence can be expressed as a specialisation of the generating series of \emph{elementary symmetric functions}. This combinatorial reformulation, which first appeared in \cite[Section 4.2]{CDMN20}, is one of the main argument which enables the extension of the higher order Chen--Stein inequalities to random variables which are not sums of independent Bernoulli variables. We detail this idea in Section \ref{sec:symmetric_functions}, and we then state our main results and the relevant hypotheses for the random models. In Section \ref{sec:fourier_inversion}, we use Fourier inversion on the unit circle $\Tor$ in order to estimate the total variation distance between a probability measure $\mu$ on $\N$ and its approximation scheme $\nu^{(r)}$ or order $r \geq 1$. Classical arguments allow us to remove most terms of the Fourier inversion formula, and what remains is a sum of integrals akin to the integral expressions of Hermite polynomials. These Hermite-like functions are studied in Section \ref{sec:hermite}, and the summation of all the estimates and the proof of the main results is performed in Section \ref{sec:riemann_sum}. Finally, in Section \ref{sec:hankel}, we revisit the proof of the Flajolet--Odlyzko transfer theorem (see \cite{FO90}) which yields the asymptotic of the coefficients of a power series with algebraic singularities. This enables us to obtain an unconditional upper bound on the total variation distance between the law of 
\begin{itemize}
    \item the number $C_n$ of cycles in a random permutation;
    \item or, the number $D_n$ of irreducible divisors in a uniformly chosen random polynomial with given degree and coefficients in a finite field
\end{itemize}
and its derived approximation scheme of order $r \geq 1$.
 This discussion relies on complex analysis arguments and integrals along Hankel contours, and a similar argument can be used for the number $\omega_n$ of prime factors of a random integer; see our Remark \ref{rem:next} at the very end of the paper.
\bigskip

\section{Symmetric functions and adapted sequences of distributions}\label{sec:symmetric_functions}
The objective of this section is to introduce all the relevant hypotheses for our main theorems, as well as a list of examples which will satisfy these hypotheses. A large part of the discussion will rely on the combinatorics of the \emph{algebra of symmetric functions} $\Sym$, for which we refer to \cite[Chapter I]{Mac95} and \cite[Chapter 2]{Mel17}. Recall that a symmetric function is a formal linear combination $f=\sum_{I=(i_1,\ldots,i_r)} f_I \,x_I$ of monomials $x_I=x_{i_1}x_{i_2}\cdots x_{i_r}$ with:
\begin{itemize}
    \item the indices $I$ in $\bigsqcup_{r=0}^\infty (\N^*)^r$;
    \item the variables $x_1,x_2,\ldots$ forming an infinite commutative sequence;
    \item the coefficients $f_I$ in some field, say $\R$;
    \item $\deg f = \sup\{|I|\,\,\text{with }f_I \neq 0\} < +\infty$;
    \item $f$ invariant by any permutation of the variables: for any $\sigma \in \sym(\infty) = \bigcup_{n=1}^\infty \sym(n)$, $f^\sigma=f$.
\end{itemize}
We shall use two important algebraic bases of $\Sym$ over $\R$: the Newton \emph{power sums}
$$\frakp_{k \geq 1} = \sum_{i=1}^\infty (x_i)^k$$
and the \emph{elementary symmetric functions}
$$\frake_{k \geq 1} = \sum_{1\leq i_1<i_2<\cdots < i_k} x_{i_1}x_{i_2}\cdots x_{i_k}.$$
Thus, $\mathrm{Sym} = \R[\frakp_1,\frakp_2,\ldots] = \R[\frake_1,\frake_2,\ldots]$. The change of basis formula between power sums and elementary symmetric functions is encoded by the two generating series $\frakP(z) = \sum_{k=1}^\infty \frac{\frakp_k}{k}\,z^k$ and $\frakE(z) = 1+\sum_{k=1}^\infty \frake_k\,z^k$:
\begin{align*}
\frakE(z) = \prod_{i=1}^\infty (1+x_iz) = \exp(\sum_{i=1}^\infty \log(1+x_iz)) = \exp(-\sum_{k=1}^\infty \sum_{i=1}^\infty \frac{(-x_iz)^k}{k})=\exp(-\frakP(-z)).
\end{align*}
If $A=\{a_1,a_2,\ldots\}$ is a summable family and $f$ is a symmetric function, we shall denote $f(A)$ the real number obtained by replacing the variables $x_{i}$ by the $a_i$'s (setting $x_i=0$ if $A$ is finite and $i>|A|$). This always gives a convergent power series, because $f$ is a polynomial in the power sums, and 
$$\frakp_k(A) = \sum_{i\geq 1} (a_i)^k \,\text{ is absolutely convergent for any }k\geq 1.$$
The map $f \in \Sym \mapsto f(A) \in \R$ is a morphism of real algebras; it is also called a \emph{specialisation} of the algebra of symmetric functions. More generally, we call specialisation of $\Sym$ any morphism of algebras from $\Sym$ to $\R$; such a morphism does not necessarily come from a summable family $A=\{a_1,a_2,\ldots\}$. Given the generating functions $\frakE(z)$ and $\frakP(z)$ of the elementary symmetric functions and of the power sums and a specialisation $A$ of $\Sym$, we shall denote $\frakE(A,z)$ and $\frakP(A,z)$ the corresponding analytic functions of the variable $z$, assuming the convergence of these power series.
\medskip

\subsection{Sums of independent Bernoulli variables}\label{sub:sum_bernoulli}
Let us consider as in the introduction a sum $X_n = \sum_{i=1}^n Y_i$ of independent Bernoulli variables, with $Y_i \sim \mathrm{Be}(p_i)$. We are going to explain how to compute the coefficients $b_{s,n}$ of the residue of deconvolution
$\psi_n(\xi) = \esper[\E^{\I \xi X_n}]\,\E^{-\lambda_n(\E^{\I \xi}-1)}$, with $\lambda_n = \sum_{i=1}^n p_i$. To begin with, let us remark that the Fourier transform of $X_n$ is a specialisation of the generating series $\frakE(z)$ of the elementary symmetric functions. Indeed, it is obtained by taking  $z=\E^{\I \xi}-1$ and the alphabet $\{p_1,p_2,\ldots,p_n\}$:
\begin{align*}
\esper[\E^{\I \xi X_n}] &= \prod_{i=1}^n \esper[\E^{\I \xi Y_i}] = \prod_{i=1}^{n} (1+p_i(\E^{\I \xi}-1)) = \frakE(\{p_1,\ldots,p_n\},\E^{\I \xi }-1) \\ 
&= \exp(-\frakP(\{p_1,\ldots,p_n\},1-\E^{\I \xi})) = \exp(\sum_{k=1}^\infty \frac{(-1)^{k-1}\,\frakp_k(p_1,\ldots,p_n)}{k}\,(\E^{\I \xi}-1)^k).
\end{align*}
Dividing by $\E^{\lambda_n(\E^{\I \xi}-1)}$ amounts to remove the term of order $k=1$ from the exponential, so
$$\psi_n(\xi) = \exp(\sum_{k=2}^\infty \frac{(-1)^{k-1}\,\frakp_k(p_1,\ldots,p_n)}{k}\,(\E^{\I \xi}-1)^k).$$
The coefficients $b_{s,n}$ are then obtained by expanding the exponential series. There are two ways to perform this computation:
\begin{itemize}
    \item specialisation with $\frakp_1=0$. Given a countable family of real numbers $A=\{a_1,a_2,\ldots\}$ with $\sum_{i\geq 1}(a_i)^2 < +\infty$, we define a morphism of real algebras $f \in \Sym \mapsto f(A') \in \R$ by setting:
    $$\frakp_1(A')=0\qquad;\qquad \frakp_{k \geq 2}(A') = \frakp_k(A) = \sum_{i \geq 1}(a_i)^k.$$
    Since $(\frakp_{k})_{k \geq 1}$ is an algebraic basis of $\Sym$, the formul{\ae} above entirely determine the specialisation $A'$ ($A'$ is also sometimes called a \emph{virtual alphabet}). Now, with $A = \{p_1,\ldots,p_n\}$, we have
    $$\psi_n(\xi) = \exp(\sum_{k=1}^\infty \frac{(-1)^{k-1}\,\frakp_k(A')}{k}\,(\E^{\I \xi}-1)^k) = 1 + \sum_{k=1}^\infty \frake_k(A')\,(\E^{\I \xi}-1)^k=\frakE(A',\E^{\I\xi}-1),$$
    so $b_{s,n} = \frake_s(\{p_1,\ldots,p_n\}')$ for any $s \geq 1$.

    \item inclusion-exclusion formula with the true elementary symmetric functions. Let us make the previous argument a bit more explicit. If we expand the exponential generating series $\exp(-\frakP(-z))$, we get:
    $$\frake_k = \sum_{\lambda \in \mathfrak{Y}(k)} \frac{(-1)^{k-\ell(\lambda)}}{z_\lambda}\,\frakp_{\lambda},$$
    where the sum runs over  the set $\mathfrak{Y}(k)$ of integer partitions $\lambda=(\lambda_1\geq \lambda_2 \geq \cdots \geq \lambda_{\ell(\lambda)} \geq 1)$ of size $k = \sum_{i=1}^{\ell(\lambda)} \lambda_i$; $\frakp_{\lambda} = \frakp_{\lambda_1}\,\frakp_{\lambda_2}\cdots \frakp_{\ell(\lambda)}$ for any integer partition $\lambda$; and $z_\lambda$ is a combinatorial coefficient, such that $\frac{k!}{z_\lambda}$ is the number of permutations with size $k$ and with cycle-type $\lambda$. A multiplicative expression of $z_\lambda$ in terms of the parts of the integer partition $\lambda$ is provided by \cite[Chapter I, Equation (2.14)]{Mac95}. By specialisation of the formula above with respect to the two formal alphabets $A = \{p_1,\ldots,p_n\}$ and $A'$, we get:
    \begin{align*}
    \frake_k(A) &= \sum_{\lambda \in \mathfrak{Y}(k)} \frac{(-1)^{k-\ell(\lambda)}}{z_\lambda}\,\frakp_{\lambda}(p_1,\ldots,p_n) ;\\
    b_{k,n}=\frake_k(A')&= \sum_{\substack{\lambda \in \mathfrak{Y}(k) \\ \lambda_{\ell(\lambda)}\geq 2}} \frac{(-1)^{k-\ell(\lambda)}}{z_\lambda}\,\frakp_{\lambda}(p_1,\ldots,p_n).
    \end{align*}
    In particular, 
    \begin{align*}
    b_{1,n}&=0\quad;\quad b_{2,n} = -\frac{1}{2}\,\frakp_2(p_1,\ldots,p_n)\quad;\quad b_{3,n} = \frac{1}{3}\,\frakp_3(p_1,\ldots,p_n) \\ 
    b_{4,n} &= \frac{(\frakp_2(p_1,\ldots,p_n))^2}{8} - \frac{\frakp_4(p_1,\ldots,p_n)}{4}.
    \end{align*}
    The removal of the integer partitions with parts of size $1$ which is performed when going from $\frake_k(A)$ to $\frake_k(A')$ also results from an inclusion-exclusion; hence, it is easy to see from the formul{\ae} above that for any $s \geq 1$,
    $$b_{s,n}
 = \sum_{t=0}^s \frac{(-1)^{t}}{t!}\,(\frake_1(p_1,\ldots,p_n))^{t}\,\frake_{s-t}(p_1,\ldots,p_n)=\sum_{t=0}^s \frac{(-1)^{t}}{t!}\,(\lambda_n)^{t}\,\frake_{s-t}(p_1,\ldots,p_n),$$
 with by convention $\frake_0(p_1,\ldots,p_n)=1$.
\end{itemize}

\begin{remark}
The expression of $b_{s,n}$ in terms of the elementary symmetric functions $\frake_t(p_1,\ldots,p_n)$ can be used to prove that each coefficient $b_{s,n}$ is a polynomial of total degree $s$ in the moments $M_{k,n}=\esper[(X_n)^k]$, with $M_{k,n}$ considered to be of degree $k$. Indeed,
$$\esper[(X_n)^k] = \sum_{l=1}^k l!\,\genfrac\{\}{0pt}{0}{k}{l}\,\frake_{l}(p_1,\ldots,p_n),$$
where $\genfrac\{\}{0pt}{1}{k}{l}$ is the Stirling number of the first kind, which counts set partitions of $\lle 1,k\rre$ in $l$ parts. The formula above can be inverted in order to express the elementary symmetric functions of the probabilities $p_i$ in terms of the moments of $X_n$. Therefore, the coefficients $b_{s,n}$ with $s \leq r$ and the approximation scheme of order $r \geq 1$ of the law of $X_n$ depends only on the $r$ first moments of $X_n$, and not on the individual probabilities $p_i$. For instance, 
\begin{align*}
b_{2,n} &= \frac{1}{2}\,(M_{2,n}-M_{1,n}-(M_{1,n})^2);\\ 
b_{3,n} &=\frac{1}{2}\,((M_{1,n})^2-M_{2,n}-M_{2,n}M_{1,n}) + \frac{1}{3}\,(M_{1,n} + (M_{1,n})^3 ) + \frac{1}{6}\,M_{3,n}.
\end{align*}
This property is important for simulations, if one wants to approximate the law of $X_n$ from a dataset instead of the list of the probabilities $p_i$. In a companion paper \cite{MNV22}, we investigate the applications of the theory developed in the present article to credit risk models. We refer to \cite[Remark 4.3 and Appendix D]{MNV22} for the formula of change of basis between the coefficients $b_{s,n}$ and the moments of $X_n$.
\end{remark}

Before going on, let us analyse the radius of convergence of the generating series that we have manipulated above.
\begin{proposition}\label{prop:cauchy_formula}
Let $\{p_1,\ldots,p_n\}$ be a set of probabilities in $(0,1)$, $\lambda_n=\sum_{i=1}^n p_i$, $X_n = \sum_{i=1}^n \mathrm{Be}(p_i)$, and 
$\Psi_n(w) = \esper[w^{X_n}]\,\E^{-\lambda_n(w-1)}$.
With our previous notations, $\psi_n(\xi)=\Psi_n(\E^{\I \xi})$. 
\begin{enumerate}
    \item If $b_{s,n}=\frake_s(\{p_1,\ldots,p_n\}')$, then we have 
    $$\Psi_n(w) = \frakE(\{p_1,\ldots,p_n\}',w-1) =1+\sum_{s=1}^\infty b_{s,n}\,(w-1)^s,$$
and the right-hand side converges on the whole complex plane; thus, $\Psi_n$ is an entire function.
\item More precisely, if $(\sigma_n)^2 = \sum_{i=1}^n (p_i)^2$, then for any $s \geq 2$,
$$|b_{s,n}|\leq \left(\frac{\E(\sigma_n)^2}{s}\right)^{\!\frac{s}{2}}.$$
\end{enumerate}
\end{proposition}

\begin{proof}
Since $\Psi_n(w) = \prod_{i=1}^n ((1+p_i(w-1))\,\E^{-p_i(w-1)})$ is a product of entire functions, it is indeed well defined and convergent on the whole complex plane, and its coefficients $b_{s,n}$ come from the previous computations. Notice that we have in particular $b_{1,n} = \frake_1(\{p_1,\ldots,p_n\}') = 0$. In order to prove the upper bound on the coefficients $b_{s \geq 2,n}$, we shall use the following elementary inequality:
$$\forall z \in \C,\,\,|(1+z)\,\E^{-z}| \leq \E^{\frac{|z|^2}{2}}.$$ 
Indeed, if $z=x+\I y$, let us fix $x \in \R$ and study as a function of $y$
$$\frac{|(1+z)\,\E^{-z}|^2}{\E^{|z^2|}} = ((1+x)^2 + y^2)\,\E^{-2x - x^2-y^2}.$$
The derivative with respect to $y$ of this function vanishes if and only if $y=0$ or $1 = (1+x)^2 + y^2$. If $x \in (-2,0)$, then the two maxima of the function are equal to $1$, whereas if $x\leq -2$ or $x\geq 0$, then the unique maximum is attained at $y=0$ and is equal to $(1+2x+x^2)\, \E^{-2x-x^2} \leq \E^{2x+x^2-2x-x^2}=1$. \medskip

Now, by using the Cauchy formula with a circle of radius $R = \frac{\sqrt{s}}{\sigma_n}$, we get:
\begin{align*}
|b_{s,n}| &= \left|\frac{1}{2\I\pi} \oint \prod_{i =1}^n\big((1+p_iz)\,\E^{-p_iz}\big) \,\frac{\!\DD{z}}{z^{s+1}}\right| \leq \frac{\E^{\frac{(\sigma_n)^2R^2}{2}}}{R^s} = \left(\frac{\E(\sigma_n)^2}{s}\right)^{\!\frac{s}{2}}.
\end{align*}
This fast decay of the coefficients $b_{s,n}$ will turn out to be an essential tool in the proofs of our estimates.
\end{proof}
\medskip

\subsection{Models with Bernoulli asymptotics}\label{sub:bernoulli_like_residue}
Consider more generally a countable family of real numbers $A=\{a_1,a_2,\ldots\}$ with $\sigma^2=\sum_{i\geq 1}(a_i)^2 < +\infty$. Then, the result from Proposition \ref{prop:cauchy_formula} extends readily to
\begin{equation}
\frakE(A',w-1) = \prod_{i \geq 1} (1+a_i(w-1))\,\E^{-a_i(w-1)} = 1+\sum_{s=2}^\infty\frake_s(A')\,(w-1)^s.\label{eq:bernoulli_like}
\end{equation}
Thus, the function $\frakE(A',\cdot)$ is an entire function on $\C$, and its coefficients $\frake_s(A')$ satisfy the inequality:
$$\forall s \geq 2,\,\,|\frake_s(A')| \leq \left(\frac{\E\sigma^2}{s}\right)^{\!\frac{s}{2}}.$$
This can be seen by taking the limit of the corresponding result for the truncated finite alphabets $A_n=\{a_1,a_2,\ldots,a_n\}$. 
\begin{definition}
Let $(X_n)_{n \in \N}$ be a sequence of random variables with values in $\N$. We say that the sequence converges mod-Poisson with parameters $(\lambda_n)_{n \in \N}$ and with \emph{Bernoulli asymptotics} if there exists a square-summable family $A=\{a_1,a_2,\ldots\}$ such that:
$$\esper[w^{X_n}]\,\E^{-\lambda_n(w-1)} \to_{n \to \infty} \frakE(A',w-1),$$
where $\frakE(A',\cdot)$ is the entire function defined by Equation \eqref{eq:bernoulli_like}, and where the convergence happens locally uniformly on a disc of radius $r>1$.
\end{definition}

In \cite[Section 4, Table 4.1]{CDMN20}, a list of examples of mod-Poisson models with Bernoulli asymptotics is provided; the computations of the corresponding generating series rely on relatively elementary combinatorial or algebraic arguments. The remainder of this Subsection is devoted to recalling these important examples; the proofs of the corresponding mod-Poisson convergence results will be revisited in Section \ref{sec:hankel} in order to get explicit estimates of the remainders.

\subsubsection{Number of cycles of a random permutation}
Let $(\theta_{k \geq 1})_{k \geq 1}$ be a sequence of positive parameters, and let $\Theta$ be the specialisation of the algebra $\Sym$ defined by $\frakp_k(\Theta)=\theta_k$ for any $k \geq 1$. We also introduce the \emph{homogeneous symmetric functions} 
$$\frakh_k = \sum_{1\leq i_1\leq i_2 \leq \cdots \leq i_k} x_{i_1}x_{i_2}\cdots x_{i_k};$$
they form another algebraic basis of $\Sym$, and they differ from the elementary symmetric functions by allowing equalities between the indices $i_1,\ldots,i_k$. The generating series of the homogeneous symmetric functions $\frakH(z) = 1+\sum_{k=1}^\infty \frakh_k\,z^k$ is related to the generating series $\frakP(z)$ by:
$$\frakH(z) = \prod_{i=1}^\infty \frac{1}{1-x_iz}=\exp(-\sum_{i=1}^\infty \log(1-x_iz))=\exp(\sum_{k=1}^\infty \sum_{i=1}^\infty \frac{(x_iz)^k}{k})=\exp(\frakP(z)).$$
By expanding the exponential, we get an expression of the symmetric function $\frakh_k$ in terms of the power sums, which is very similar to what we saw previously with elementary symmetric functions:
$$\frakh_k = \sum_{\lambda \in \mathfrak{Y}(k)}\frac{1}{z_\lambda}\,\frakp_\lambda.$$
Let $\proba_{n,(\theta_{k \geq 1})_{k \geq 1}}$ be the probability measure on the symmetric group $\sym(n)$ which gives to a permutation $\sigma$ with $m_1(\sigma)$ cycles of length $1$, $m_2(\sigma)$ cycles of length $2$, etc. a probability proportional to $\prod_{k \geq 1} (\theta_k)^{m_k(\sigma)}$:
$$\proba_{n,(\theta_k)_{k \geq 1}}[\sigma] = \frac{1}{Z_{n,(\theta_k)_{k \geq 1}}} \,\prod_{k \geq 1} (\theta_k)^{m_k(\sigma)}.$$
The normalisation constant is easy to compute. Indeed, 
$$Z_{n,(\theta_k)_{k \geq 1}} = \sum_{\sigma \in \sym(n)}(\theta_k)^{m_k(\sigma)} = \sum_{\lambda \in \mathfrak{Y}(n)} \frac{n!}{z_\lambda}\,\frakp_\lambda(\Theta) = n!\,\frakh_n(\Theta).$$
These \emph{weighted measures} $\proba_{n,(\theta_k)_{k \geq 1}}$ have been studied in \cite{BU09,BU11,BUV11,EU12,NZ13}, in connection with models of spatial random permutations. If $(\theta_k)_{k \geq 1}$ is the constant sequence equal to $\theta >0$, then we recover the \emph{Ewens measure} with parameter $\theta$, for which $\frakh_n(\Theta) = \prod_{i=1}^n (1+\frac{\theta-1}{i})$.
Consider now the random variable 
$$C_n = \big(\text{number of disjoint cycles of }\sigma_n \sim \proba_{n,(\theta_k)_{k \geq 1}}\big).$$
Its generating series is:
\[\esper[w^{C_n}] = \frac{1}{n!\,\frakh_n(\Theta)} \sum_{\sigma \in \sym(n)} \left(\prod_{k \geq 1}(w\theta_k)^{m_k(\sigma)}\right) = \frac{\frakh_n(w\Theta)}{\frakh_n(\Theta)},\label{eq:gen_cycles}\tag{genC}\]
where $w\Theta$ denotes the specialisation of $\Sym$ given by $\frakp_{k \geq 1}(w\Theta) = w\theta_k$.
\begin{proposition}\label{prop:random_permutations}
Suppose that the sequence of parameters $(\theta_k)_{k \geq 1}$ yields a specialisation $\Theta$ of $\Sym$ such that
$$\frakP(\Theta,z) = \sum_{k=1}^\infty \frac{\theta_k}{k}\,z^k $$
has the following properties:
\begin{enumerate}
    \item The generating series $\frakP(\Theta,z)$ is holomorphic on a domain 
    $$\Delta(m,M,\phi)=\{z\in \C\,\text{ such that }|z|< M,\,\,z\neq m,\,\,|\arg(z-m)|>\phi\}$$
     with $0<m<M$, $0<\phi<\frac{\pi}{2}$.
     
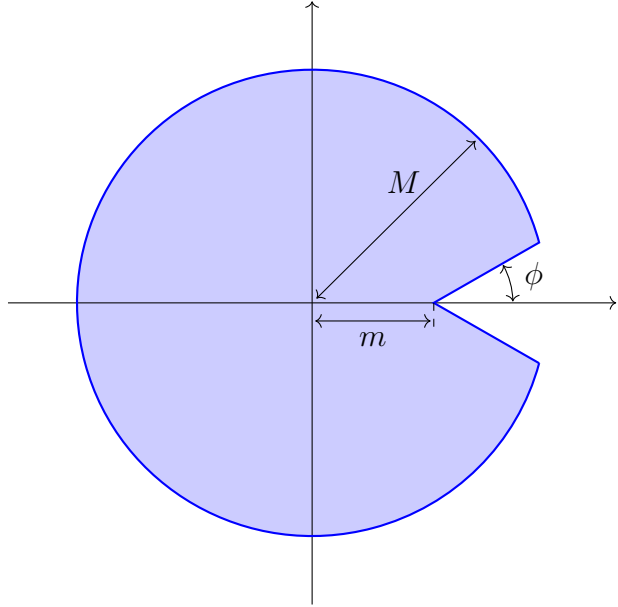
\begin{center}
\begin{tikzpicture}[scale=0.8]
\fill [white!80!blue,shift={(2,0)}]  (-30:2) -- (0,0) -- (30:2) arc (15:345:3.864) -- (-30:2);
\draw [->] (-5,0) -- (5,0);
\draw [->] (0,-5) -- (0,5);
\draw [thick,blue,shift={(2,0)}]  (-30:2) -- (0,0) -- (30:2) arc (15:345:3.864) -- (-30:2);
\draw [<->,shift={(2,0)}] (1:1.3) arc (1:29:1.3);
\draw [shift = {(2,0)}] (15:1.7) node {$\phi$};
\draw [<->,shift={(0,-0.3)}] (0.05,0) -- (1.95,0);
\draw [dashed] (2,0) -- (2,-0.4);
\draw (1,-0.6) node {$m$};
\draw [<->] (45:0.1) -- (45:3.8);
\draw (53:2.5) node {$M$};
\end{tikzpicture}
\end{center}
     \item Around $z=m$, $\frakP(\Theta,z)$ has a logarithmic singularity, and
     $$\frakP(\Theta,z) = \theta\,\log(\frac{1}{1-\frac{z}{m}}) + K + O(|z-m|)$$
     with $\theta>0$.
\end{enumerate}
We set $$\gamma_\theta = \sum_{n=1}^\infty \frac{\theta}{n+\theta-1} - \theta \log(1+\frac{1}{n});$$ notice that $\gamma_1=\gamma$ is the Euler--Mascheroni constant.
Then, the sequence of numbers of disjoint cycles $(C_n)_{n \in \N}$ converges mod-Poisson with parameters $\lambda_n=\theta\,\log n + K+\gamma_\theta$ and Bernoulli asymptotics. The limiting alphabet $A_\theta$ is 
$$A_\theta = \left\{1,\frac{\theta}{\theta+1},\frac{\theta}{\theta+2},\ldots\right\}$$
Moreover, the convergence happens locally uniformly on $\C$ at speed $O(n^{-1})$:
$$\esper[w^{C_n}]\,\E^{-(\theta \log n + K+\gamma_\theta)(w-1)} =\frakE(A_\theta',w-1) + O(n^{-1}).$$
\end{proposition}

A proof of this result is given in \cite[Lemma 4.1]{NZ13}; we shall see in Section \ref{sec:hankel} that it is a particular case of transfer results for generating series with algebraico-logarithmic singularities (see \cite{FO90,Hwa99} and the discussion of \cite[Section 4.4]{CDMN20}). Let us remark that the form asked for the singularity of $\frakP(\Theta,z)$ is inspired by the case of Ewens measures: indeed, we then have $\frakP(\Theta,z) = \theta \log(\frac{1}{1-z})$ and $K=0$.

\subsubsection{Number of irreducible factors of a random polynomial}
Let $q=p^e$ be a prime power, and $f_n$ be a random polynomial chosen uniformly among the $q^n$ monic polynomials with degree $n$ in $\For_q[X]$. We denote
$$D_n =D(f_n)= \big(\text{number of distinct irreducible divisors of }f_n\big),$$
the irreducible factors being counted without multiplicity. If $\Irr(\For_q)$ denotes the set of all irreducible monic polynomials in $\For_q[X]$, then we have the following identities of generating series:
\begin{align*}
\sum_{n=0}^\infty (qz)^n \,\esper[w^{D_n}] &= \sum_{\substack{P \in \For_q[X]\\ P \text{ monic}}} w^{D(P)}\,z^{\deg P} = \prod_{P \in \Irr(\For_q)} (1+wz^{\deg P} + wz^{2\deg P} + \cdots) \\ 
&=\prod_{P \in \Irr(\For_q)} \left(1+w\,\frac{z^{\deg P}}{1-z^{\deg P}}\right)=\prod_{P \in \Irr(\For_q)} \frac{1-(1-w)z^{\deg P}}{1-z^{\deg P}} \\ 
&=\exp(\sum_{P \in \Irr(\For_q)} \log(1-(1-w)z^{\deg P}) - \log(1-z^{\deg P})) \\ 
&= \exp(\sum_{m=1}^\infty \sum_{k=1}^\infty \frac{ I_{q}(m)\,z^{km}(1-(1-w)^k)}{k}) \\ 
&= \exp(\sum_{n=1}^\infty \frac{z^n}{n} \sum_{k\mathrel{|}n} \frac{n}{k}\, I_{q}\!\left(\frac{n}{k}\right)\,(1-(1-w)^k))\\
&= \exp(\sum_{n=1}^\infty \frac{z^n}{n}\,(\widetilde{I}_q * J_w)(n))
\end{align*}
where $I_q(n) = \card\{P \in \Irr(\For_q)\,|\,\deg P =n\}$, $\widetilde{I}_q(n)=n\,I_q(n)$, $J_w(n) = 1-(1-w)^n$, and $*$ denotes the operation of convolution on arithmetic functions:
$$\forall n \geq 1,\,\,\,(f*g)(n) =\sum_{k\mathrel{|}n} f(k)\,g\!\left(\frac{n}{k}\right).$$
The number of irreducible polynomials with a given degree is well known to be given by Gauss' formula:
$$\widetilde{I}_q(n) = (\mu * q^\bullet)(n) ,$$
where $\mu$ is the arithmetic Möbius inversion function. In particular, $I_q(n) \leq \frac{q^n}{n}$. Then,
$$(\widetilde{I}_q * J_w)(n) = (\mu* q^\bullet * (1-(1-w)^\bullet))(n).$$
In particular, if $w=1$, then $(\widetilde{I}_q * J_1)(n) = (\mu*1*q^{\bullet})(n) = (q^\bullet)(n)=q^n$, and $\exp(\sum_{n=1}^\infty \frac{(qz)^n}{n}) = \frac{1}{1-qz} = \sum_{n=0}^\infty (qz)^n$. Therefore, if we define a specialisation of $\mathrm{Sym}$ by setting
$$\frakp_k(B_{q,w}) = (\mu* q^\bullet * (1-(1-w)^\bullet))(k)$$
for any $k \geq 1$, then we obtain an expression of the generating series $\esper[w^{D_n}]$ similar to the one of the previous paragraph (Equation \eqref{eq:gen_cycles}):
\[
    \esper[w^{D_n}] =\frac{\frakh_n(B_{q,w})}{\frakh_n(B_{q,1})}.\label{eq:gen_divisors}\tag{genD}
\]
The analysis of the bivariate generating series leads to the following result of mod-Poisson convergence (see \cite[Example 4.5 and Theorem 4.6]{CDMN20}), which is a function field analogue of the Erd\H{o}s--Kac asymptotics of the sequence $(\omega_n)_{n \in \N}$.
\begin{proposition}\label{prop:random_polynomials}
We fix a prime power $q=p^e$ and we denote $$R_q=\sum_{k=2}^\infty \frac{\mu(k)}{k}\log(\frac{1}{1-q^{1-k}}).$$ The sequence of numbers of distinct irreducible factors $(D_n)_{n \in \N}$ converges mod-Poisson with parameters $\lambda_n=\log n + R_q+\gamma$ and Bernoulli asymptotics. The limiting alphabet $A_q$ is 
$$A_q=\left\{1,\frac{1}{2},\frac{1}{3},\ldots\right\}\sqcup \left\{\frac{1}{q^{\deg P}},\,\,P\in \Irr(\For_q)\right\}.$$
Moreover, the convergence happens locally uniformly on the disc with radius $q$ at speed $O(n^{-1})$:
$$\esper[w^{D_n}]\,\E^{-(\log n + R_q+\gamma)(w-1)} =\frakE(A_q',w-1)\, + O(n^{-1}).$$
\end{proposition}

\subsubsection{Number of prime divisors of a random integer}
The discussion of Subsection \ref{sub:erdos_kac} can be put in the framework of mod-Poisson convergent sequences with Bernoulli asymptotics. Indeed, the complex version of Equation \eqref{eq:selberg_delange} is valid, so we get:
\begin{proposition}
 The sequence of numbers of distinct prime divisors $(\omega_n)_{n \in \N}$ converges mod-Poisson with parameters $\lambda_n=\log \log n +\gamma$ and Bernoulli asymptotics. The limiting alphabet $A_\omega$ is 
$$A_\omega=\left\{1,\frac{1}{2},\frac{1}{3},\ldots\right\}\sqcup \left\{\frac{1}{p},\,\,p \in \mathbb{P}\right\}.$$
Moreover, the convergence happens locally uniformly on the complex plane at speed $O((\log n)^{-1})$:
$$\esper[w^{\omega_n}] \,\E^{-(\log \log n +\gamma)(w-1)} = \frakE(A_\omega',w-1) + O((\log n)^{-1}).$$
 \end{proposition} 
\noindent This mod-Poisson convergence is a consequence of a Tauberian theorem which yields estimates of the sum $\sum_{k=1}^n w^{\omega(k)}$, by using properties of the $L$-series
\begin{align*}
\sum_{n = 1}^\infty \frac{w^{\omega(n)}}{n^s} &= \prod_{p \in \mathbb{P}} \left(1+\frac{w}{p^s}+\frac{w}{p^{2s}}+\cdots\right)  = \prod_{p \in \mathbb{P}} \left(1+\frac{w}{p^s-1}\right);
\end{align*}
see our Remark \ref{rem:next} at the end of the article.
\medskip

\subsection{Estimates of the total variation distance}
We are now ready to state our main theorems. The first result estimates the quality of the approximation scheme of order $r$ of a integer-valued random variable $X$ whose deconvolution residue $\esper[\E^{\I \xi X}]\,\E^{-\lambda(\E^{\I \xi}-1)}$ can be expanded on the torus as a power series in $z=\E^{\I\xi}-1$, with the coefficients of the series that satisfy the same kind of inequality as in Proposition \ref{prop:cauchy_formula}.

\begin{main}\label{main:A}
Let $X$ be a integer-valued random variable such that
$$\psi(\xi) = \esper[\E^{\I \xi X}]\,\E^{-\lambda(\E^{\I \xi}-1)} = 1+\sum_{s=2}^\infty b_{s}\,(\E^{\I\xi}-1)^s\qquad;\qquad |b_{s}| \leq \left(\frac{\tau}{\sqrt{s}}\right)^{s}$$
for some positive parameters $\lambda$ and $\tau$. We denote $\nu^{(r)}$ the approximation scheme of order $r\geq 1$ of the distribution $\mu$ of $X$:
$$\widehat{\nu}(\xi)\, \E^{-\lambda(\E^{\I \xi}-1)}=1+\sum_{s=2}^r b_s\,(\E^{\I \xi}-1)^s.$$
Then, there exists two universal constants $C$ and $D$ such that, if $\eps = \frac{D\tau}{\sqrt{\lambda}}<1$, then
$$\dtv(\mu,\nu^{(r)}) \leq C\,\eps^{r+1} .$$
If $D=4$, then one can take $C\leq 570$.
\end{main}

\begin{remark}
The reason why we do not insist on the precise value of the constants $C$ and $D$ in Theorem \ref{main:A} is the following. During the proof of Theorem \ref{main:A} (Sections \ref{sec:fourier_inversion} to \ref{sec:riemann_sum}), we shall add numerous quantities which will depend on the index $r$ of approximation, and we shall then use upper bounds on these quantities which are independent from $r \geq 1$. In almost every case, the order of approximation $r=1$ yields by far the worst estimates; and assuming that $r$ is larger (for instance, larger than $10$) yields much better constants. So, one can state a version of Theorem \ref{main:A} with much smaller universal constants $C$ and $D$ if one replaces the hypothesis $r\geq 1$ by $r \geq 10$. In the following, we tried to make the computation of the constants in the upper bounds easy to track; thus, they are easily improved upon with additional assumptions on $r$ (and possibly on $\tau$ and $\lambda$).
\end{remark}

An immediate consequence of Theorem \ref{main:A} is the following higher-order Chen--Stein inequality, which regards the approximation of the distribution of a sum of independent Bernoulli variables:

\begin{main}\label{main:B}
Let $(p_i)_{i \geq 1}$ be a sequence of probabilities in $(0,1)$, and $X_n = \sum_{i=1}^n \mathrm{Be}(p_i)$ be the sum of independent Bernoulli variables. If 
$$\lambda_n = \sum_{i=1}^n p_i \qquad;\qquad(\sigma_n)^2=\sum_{i=1}^n (p_i)^2\qquad;\qquad \lambda_n>16\,\E\,(\sigma_n)^2$$ and if $\nu_n^{(r)}$ is the approximation scheme of order $r\geq 1$ of the distribution $\mu_n$ of $X_n$, then there exists a universal constant $C \leq 570$ such that
$$\dtv(\mu_n,\nu_n^{(r)}) \leq C\,(\eta_n)^{r+1},\quad\text{with }\eta_n = \frac{4\sqrt{\E}\,\sigma_n}{\sqrt{\lambda_n}}<1.$$
\end{main}

\begin{proof}
This follows readily from Proposition \ref{prop:cauchy_formula}, which yields the estimate $|b_{s,n}| \leq (\frac{\tau_n}{\sqrt{s}})^{s}$ with $\tau_n = \sqrt{\E}\,\sigma_n$.
\end{proof}

\begin{corollary}\label{cor:derived_scheme}
Let $(p_i)_{i \geq 1}$ be a non-increasing sequence of probabilities in $(0,1)$ such that $\sum_{i=1}^\infty (p_i)^2 = \sigma^2 < +\infty$. We denote $\mu_n$ the distribution of $X_n = \sum_{i=1}^n \mathrm{Be}(p_i)$, $\lambda_n = \sum_{i=1}^n p_i$, and $\nu_{n,*}^{(r)}$ the derived approximation scheme of order $r \geq 1$:
$$\widehat{\nu}_{n,*}^{(r)}(\xi) = \E^{\lambda_n(\E^{\I \xi}-1)} \left(1+\sum_{s=2}^r b_s\,(\E^{\I \xi}-1)^s\right),\quad \text{with }b_s=e_s(\{p_1,p_2,\ldots\}').$$
If $\lambda_n > 16\,\E\,\sigma^2$, then
$$\dtv(\mu_n,\nu_{n,*}^{(r)}) \leq C\,(\eta_n)^{r+1} + (r^2+(2\lambda_n+1)r)\,\left(\sum_{s=2}^r (2\sigma)^{s-2}\right)\,r_n$$
with $\eta_n = \frac{4\sqrt{\E}\,\sigma}{\sqrt{\lambda_n}}$, and $r_n = \sum_{i>n} (p_i)^2$.
\end{corollary}

\begin{example}
Suppose that $p_i=\frac{1}{i}$. Notice then that $X_n$ has the distribution of the number of cycles of a uniform random permutation in $\sym(n)$. We have $\sigma^2 = \frac{\pi^2}{6}$, $\log n + \gamma \leq \lambda_n \leq \log n +\gamma+\frac{1}{2n}$, and $r_n = \sum_{i>n} \frac{1}{i^2} \leq \frac{1}{n}$. Therefore, setting $h=\sqrt{\frac{2}{3}}\,\pi$ and $A_\Gamma=\{1,\frac{1}{2},\frac{1}{3},\ldots\}$, we get that for $r \geq 1$, the derived scheme of approximation $\nu_{n,*}^{(r)}$ defined by
$$\widehat{\nu}_{n,*}^{(r)}(\xi) = \E^{\lambda_n(\E^{\I \xi}-1)}\left(1+\sum_{s=2}^r \frake_s(A_\Gamma')\,(\E^{\I \xi}-1)^s\right)$$
satisfies:
$$\dtv(X_n,\nu_{n,*}^{(r)}) \leq C\left(\frac{2\sqrt{\E}\,h}{\sqrt{\log n}}\right)^{r+1} + 1_{(r \geq 2)}\frac{h^{r-2}}{1-h^{-1}}\,\frac{r^2+2(\log n +1 )r}{n}$$
for any $n$ such that $\log n> \frac{8\E \pi^2}{3}$. In particular, if $r\geq 1$ is fixed, then the total variation distance is a $O((\log n)^{-\frac{r+1}{2}})$; this is compatible with the asymptotic estimate from \cite[Theorem 4.3]{CDMN20}.
\end{example}

\begin{lemma}\label{lem:two_step_approximation}
Let $\mu$ and $\nu$ be two signed measures on $\N$ with Fourier transforms
\begin{align*}
\widehat{\mu}(\xi) &= \E^{\lambda(\E^{\I \xi}-1)}\,\psi(\xi) ;\\ 
\widehat{\nu}(\xi) &= \E^{\lambda(\E^{\I \xi}-1)}\,\chi(\xi), 
\end{align*}
with $\lambda>0$.  The total variation distance between $\mu$ and $\nu$ is smaller than
$$\frac{\|\psi-\chi\|_{\infty}}{2} + \frac{\pi}{2\sqrt{3}} \,(\|\psi'-\chi'\|_{\infty} + \lambda \,\|\psi-\chi\|_\infty),$$
where $\|f\|_\infty = \sup_{\xi \in \R/2\pi \Z} |f(\xi)|$ for a continuous function on the circle.
\end{lemma}

\begin{proof}
The two sequences $(\mu(n))_{n \in \N}$ and $(\nu(n))_{n \in \N}$ are summable, and they yield two functions in the Wiener algebra of absolutely convergent Fourier series:
$$\widehat{\mu}(\xi) = \sum_{n \in \N} \mu(n)\,\E^{\I n \xi}\qquad;\qquad \widehat{\nu}(\xi) = \sum_{n \in \N} \nu(n)\,\E^{\I n \xi}.$$
By the Cauchy--Schwarz inequality,
\begin{align*}
2\,\dtv(\mu,\nu) &= \sum_{n \in \N}|\mu(n)-\nu(n)| \\ 
&\leq |\mu(0)-\nu(0)| + \sqrt{\sum_{n=1}^\infty \frac{1}{n^2}}\,\sqrt{\sum_{n=1}^\infty (n|\mu(n)-\nu(n)|)^2} \\ 
&\leq \|\widehat{\mu}-\widehat{\nu}\|_{\mathscr{L}^1(\Tor)} + \frac{\pi}{\sqrt{3}} \,\|\widehat{\mu}'-\widehat{\nu}'\|_{\mathscr{L}^2(\Tor)} \\ 
&\leq \|\psi-\chi\|_{\infty} + \frac{\pi}{\sqrt{3}} \,(\|\psi'-\chi'\|_{\infty} + \lambda \,\|\psi-\chi\|_\infty)
\end{align*}
since $\|\E^{\lambda(\E^{\I \xi}-1)}\|_\infty = 1$.
\end{proof}

\begin{proof}[Proof of Corollary \ref{cor:derived_scheme}]
If $(\sigma_n)^2 = \sum_{i=1}^n (p_i)^2$, since $\sigma^2 > (\sigma_n)^2$, we know from Theorem \ref{main:B} that 
$$\dtv(\mu,\nu_n^{(r)}) \leq C\,(\eta_n)^{r+1}$$
with $\eta_n = \frac{4\sqrt{\E}\,\sigma}{\sqrt{\lambda_n}}$. Therefore, we only have to add an estimate of $\dtv(\nu_n^{(r)},\nu_{n,*}^{(r)})$. We set $A_n = \{p_1,p_2,\ldots,p_n\}$ and $A = \{p_1,p_2,\ldots\}$. The two approximation schemes of order $r$ (standard and derived) have the same parameter $\lambda_n$, and their deconvolution residues are
$$\psi_n(\xi) = 1+\sum_{s=2}^r b_{s,n}\,(\E^{\I \xi}-1)^s\qquad;\qquad \psi(\xi) = 1+\sum_{s=2}^r b_{s}\,(\E^{\I \xi}-1)^s$$
with $b_{s,n} = \frake_s(A_n')$ and $b_s = \frake_s(A')$. If $k \geq 3$, then for any $i>n \geq j$, we have $(p_i)^2(p_j)^k \geq (p_i)^k (p_j)^2$, so
\begin{align*}
\sum_{i>n \geq j} (p_i)^2(p_j)^k &\geq \sum_{i>n \geq j} (p_i)^k(p_j)^2; \\ 
\sum_{i>n,\,\,j \geq 1} (p_i)^2(p_j)^k &\geq \sum_{i>n,\,\,j \geq 1} (p_i)^k(p_j)^2; \\ 
\frac{\sum_{i>n} (p_i)^2}{\sum_{j \geq 1}(p_j)^2} &\geq\frac{\sum_{i>n} (p_i)^k}{\sum_{j \geq 1}(p_j)^k}.
\end{align*}
As a consequence, for any integer partition $\lambda=(\lambda_1\geq \lambda_2 \geq \cdots \geq \lambda_\ell)$ with size $s$ and whose parts are all larger than $2$, we have:
\begin{align*}
|\frakp_\lambda(A) - \frakp_\lambda(A_n)| &\leq \frakp_{\lambda}(A) \sum_{a=1}^\ell \left(\frac{\frakp_{\lambda_a}(A)-\frakp_{\lambda_a}(A_n)}{\frakp_{\lambda_a}(A)}\right) \\ 
&\leq \ell\, \left(\frac{\frakp_{2}(A)-\frakp_{2}(A_n)}{\frakp_{2}(A)}\right) \frakp_\lambda(A) \\ 
&\leq \frac{s}{2}\,\frac{\sigma^2 - (\sigma_n)^2}{\sigma^2} \, \sigma^s.
\end{align*}
As $\sum_{\lambda \in \mathfrak{Y}(s)} \frac{1}{z_\lambda}=1$, we obtain from this:
\begin{align*}
|b_{s}-b_{s,n}| &\leq \frac{s}{2}\,(\sigma^2 - (\sigma_n)^2) \, \sigma^{s-2}; \\ 
\|\psi-\psi_n\|_\infty &\leq  2(\sigma^2 - (\sigma_n)^2) \left(\sum_{s=2}^r s\,(2\sigma)^{s-2}\right); \\
\|\psi'-\psi_n'\|_\infty &\leq  (\sigma^2 - (\sigma_n)^2) \left(\sum_{s=2}^r s^2\,(2\sigma)^{s-2}\right).
\end{align*}
The inequality follows then immediately from Lemma \ref{lem:two_step_approximation}, since $\frac{\pi}{2\sqrt{3}} \leq 0.9069\leq 1$.
\end{proof}
\medskip

Many ingredients in the proof of Theorems \ref{main:A} and \ref{main:B} rely on estimates of the Fourier transform of $X$, and not on the fact that $X$ is a sum of independent Bernoulli variables. As a consequence, one can extend Theorem \ref{main:B} to the case of mod-Poisson convergent sequences with Bernoulli asymptotics. 

\begin{main}\label{main:C}
Let $(X_n)_{n \in \N}$ be a sequence of integer-valued random variables and $(\mu_n)_{n \in \N}$ the corresponding sequence of discrete distributions. We suppose that the sequence $(X_n)_{n\in \N}$ converges mod-Poisson with parameters $(\lambda_n)_{n \in \N}$ and Bernoulli asymptotics:
$$ \forall w \in D(0,\rho),\,\,\left|\esper[w^{X_n}]\,\E^{-\lambda_n(\E^w-1)}-\frakE(A',w-1)\right| \leq \eps_n$$
for some square-summable family $A$, some $\rho>1$ and some sequence $(\eps_n)_{n \in \N}$ going to zero. Then, with the same universal constants $C$ and $D$ as in Theorem \ref{main:A}, setting $\sigma^2 = \sum_{a \in A}a^2$ and assuming that $\sqrt{\lambda_n}>D\sqrt{\E}\,\sigma$, we have
$$\dtv(\mu_n,\nu_{n,*}^{(r)}) \leq C\,\left(\frac{D\sqrt{\E}\,\sigma}{\sqrt{\lambda_n}}\right)^{r+1}+ \eps_n\left(\frac{\rho}{\rho-1}+\lambda_n\right).$$
In particular, if the sequence $(\lambda_n)_{n \in \N}$ goes to infinity and if $\eps_n= O((\lambda_n)^{-\frac{r+3}{2}})$, then $\dtv(\mu_n,\nu_{n,*}^{(r)}) =O((\lambda_n)^{-\frac{r+1}{2}})$.
\end{main}

\begin{proof}
The proof of Theorem \ref{main:A} will never use the fact that the distribution $\mu$ of $X$ is positive; therefore, 
we can use it with the distribution $\nu_{n,*}^{(\infty)}$ defined by the Fourier transform
$$\widehat{\nu}_{n,*}^{(\infty)}(\xi) = \E^{\lambda_n(\E^{\I \xi}-1)}\left(1+\sum_{s=2}^\infty \frake_s(A')(\E^{\I \xi}-1)^s\right).$$
Hence, $\dtv(\widehat{\nu}_{n,*}^{(\infty)},\widehat{\nu}_{n,*}^{(r)}) \leq C\,(\eta_n)^{r+1}$ with $\eta_n = \frac{D\sqrt{\E}\,\sigma}{\sqrt{\lambda_n}}$, assuming that $\eta_n<1$. We then add the distance
$$\dtv(\mu_n,\widehat{\nu}_{n,*}^{(\infty)}) \leq \|\psi_n-\psi\|_\infty + \frac{\pi}{2\sqrt{3}}\,(\|\psi_n'-\psi'\|_{\infty} + \lambda_n \,\|\psi_n-\psi\|_\infty),$$
with $\psi_n(\xi) = \widehat{\mu}_n(\xi)\,\E^{-\lambda_n(\E^{\I \xi}-1)}$ and $\psi(\xi) = \frakE(A',\E^{\I \xi}-1)$. By assumption, $\|\psi_n-\psi\|_\infty \leq \eps_n$. Moreover, we have $\psi_n(\xi)=\Psi_n(\E^{\I \xi})$ and $\psi(\xi) = \Psi(\E^{\I \xi})$, with
$$\Psi_n(w) = \esper[w^{X_n}]\,\E^{-\lambda_n(\E^w-1)}\qquad;\qquad \Psi(w) = \frakE(A',w-1).$$
Therefore, we can use the Cauchy integral formula in order to control the derivatives: for any $w=\E^{\I \xi}$ on the unit circle,
$$|\psi_n'(\xi) - \psi'(\xi)| = |\Psi_n'(w)-\Psi'(w)| \leq \frac{1}{2\pi} \int_{\partial D(w,\rho-1)} \left|\frac{\Psi_n(z)-\Psi(z)}{(z-w)^2}\right| \DD{z} \leq \frac{\eps_n}{\rho-1}.$$
We conclude as in the proof of Corollary \ref{cor:derived_scheme} by replacing $\frac{\pi}{2\sqrt{3}}$ by the larger constant $1$.
\end{proof}

\begin{remark}\label{rem:positive}
Our main Theorems \ref{main:A}, \ref{main:B} and \ref{main:C} compare the probability distribution $\mu$ of an integer-valued random variable $X$ with a \emph{signed} distribution $\nu^{(r)}$ on $\N$. For instance, in the setting of Theorem \ref{main:A} and when $r=2$, an explicit formula for $\nu^{(2)}$ is:
$$\nu^{(2)}(k) = \nu^{(0)}(k)\,\left(1 + b_2\left(1-\frac{2k}{\lambda} + \frac{k(k-1)}{\lambda^2}\right)\right),$$
$\nu^{(0)}$ being the Poisson distribution with parameter $\lambda$.
In particular, $\sum_{k \in \N} \nu^{(2)}(k) = 1$, but $\nu^{(2)}$ can take negative values for $k$ large enough if $b_2<0$ (under the hypotheses of Theorems \ref{main:B} and \ref{main:C}, $b_2=-\frac{\frakp_2(A)}{2}$ is indeed negative). This possibility is a general phenomenon for the approximating distributions $\nu^{(r)}$. Let $N_r$ be the smallest integer such that 
$$\underbrace{\sum_{m\,|\,m\leq N_r \text{ and }\nu^{(r)}(m)>0} \nu^{(r)}(m)}_{\alpha_r} > \underbrace{-\sum_{m\,|\,\nu^{(r)}(m)<0} \nu^{(r)}(m)}_{\beta_r};$$
this integer exists because $\sum_{m \in \N} \nu^{(r)}(m)=1>0$. Then, one can define a \emph{positive} probability distribution $\mu^{(r)}$ by setting
 $$\mu^{(r)}(n) = \begin{cases}
  0&\text{if }n<N_r,\\
  \alpha_r-\beta_r &\text{if }n=N_r,\\
  \max(0,\nu^{(r)}(n)) &\text{if }n>N_r.
\end{cases}$$
It is easy to see that $\dtv(\mu,\mu^{(r)}) \leq \dtv(\mu,\nu^{(r)})$. Therefore, our results yield for any $r \geq 1$ a law of random variables $\mu^{(r)}$ which is a suitable approximation of the law $\mu$. Unfortunately, it seems that there is no easy way to sample a random variable under the distribution $\mu^{(r)}$. On the contrary, for any
   bounded function $f : \N \to \C$, it is easy to compute the approximation $\nu^{(r)}(f) = \sum_{k \in \N} \nu^{(r)}(k)\,f(k)$ of $\esper[f(X)]$ by means of a sampling method. Indeed, set 
  $$g(k) = f(k) + \sum_{s=2}^r b_s ((\Delta_+)^{\circ s} f) (k),$$
  where $(\Delta_+ f)(k) = f(k+1)-f(k)$ is the discrete difference operator. Then, $\nu^{(r)}(f) = \esper[g(Y)]$, where $Y$ follows a Poisson distribution with parameter $\lambda$; see \cite[Proposition 1.12]{CDMN20}.
\end{remark}
\bigskip

\section{Fourier inversion and control of the terms with large parameters}\label{sec:fourier_inversion}
Throughout this section, $\mu$ is a probability measure on $\N$, and $\lambda$ and $\tau$ are positive real numbers such that $\lambda>(4\tau)^2$ and
$$\widehat{\mu}(\xi) \, \E^{-\lambda(\E^{\I \xi}-1)} = 1+\sum_{s=2}^\infty b_{s}\,(\E^{\I \xi}-1)^s \quad \text{with }|b_{s}|\leq\left(\frac{\tau}{\sqrt{s}}\right)^s.$$
The approximating measure of order $r$ for $\mu$ is the measure $\nu^{(r)}$ with Fourier transform
$$\widehat{\nu}^{(r)}(\xi) = \E^{\lambda(\E^{\I \xi}-1)}\,\left(1+\sum_{s=2}^r b_s\,(\E^{\I \xi}-1)\right).$$
We can consider that $\nu^{(\infty)}=\mu$. In order to control $\dtv(\mu,\nu^{(r)})$ for $r \geq 1$, the basic strategy is to use the Fourier inversion formula in order to compute $\dtv(\nu^{(s)},\nu^{(s+1)})$ for $s \geq r$, and then to sum these estimates. The fast decay in $O(s^{-\frac{s}{2}})$ of the coefficients $b_s$ will exactly compensate the fast growth in $O(s^{\frac{s}{2}})$ of certain estimates. Our strategy is inspired by certain similar but less precise arguments from \cite{Hwa99}, which were already reused in \cite{CDMN20}. In particular, hereafter we split a Fourier integral in two parts at $|\xi| =  \lambda^{-\frac{1}{6}}$, and a similar splitting at $|\xi| = \lambda^{-\frac{1}{7}}$ was used in the aforementioned articles. \medskip

\subsection{Removal of the large indices of approximation} 
To start with, let us notice that 
\begin{align*}
\widehat{\nu^{(s+1)}}(\xi) - \widehat{\nu^{(s)}}(\xi) &= b_{s+1}\,\E^{\lambda(\E^{\I \xi}-1)}\, (\E^{\I \xi}-1)^{s+1};\\
\nu^{(s+1)}(k) - \nu^{(s)}(k) &= b_{s+1} \sum_{l=0}^{s+1}(-1)^{s+1-l}\binom{s+1}{l}\, \nu(k-l),
\end{align*}
where $\nu=\nu^{(1)}$ is the Poisson distribution with parameter $\lambda$. As a consequence, regardless of the value of $\lambda$, 
$$
\sum_{k \in \N} |\nu^{(s+1)}(k) - \nu^{(s)}(k)| \leq |b_{s+1}| \sum_{l=0}^{s+1} \binom{s+1}{l}\left(\sum_{k \in \N} \nu(k-l)\right) = 2^{s+1}\,b_{s+1} \leq \left(\frac{4\,\tau^2}{s+1}\right)^{\!\frac{s+1}{2}}.
$$
Most of the work hereafter consists in proving a better upper bound when $s+1$ is not too large, say smaller than $\frac{\lambda}{4}$. For the indices $s$ such that $s+1 \geq \max(\frac{\lambda}{4},r+1)$, we shall simply use the trivial upper bound above. Let us remark that the function
$$x \mapsto \left(\frac{4\,\tau^2}{x}\right)^{\frac{x}{2}}$$
attains its maximum at $x = 4\E^{-1}\,\tau^2$. Since $\lambda> 16\,\tau^2$, $\max(\frac{\lambda}{4},r+1)>4\,\tau^2 > 4\E^{-1}\,\tau^2$ and the largest term of the series $\sum_{s+1\geq \max(\frac{\lambda}{4},r+1)} (\frac{4\,\tau^2}{s+1})^{\frac{s+1}{2}}$ is the first one, and it is smaller than $(\frac{16\,\tau^2}{\lambda})^{\frac{r+1}{2}}$. 
 Moreover, the ratio between two consecutive terms of the series is:
$$\frac{(\frac{4\,\tau^2}{s+2})^{\frac{s+2}{2}}}{(\frac{4\,\tau^2}{s+1})^{\frac{s+1}{2}}} = \sqrt{\frac{4\,\tau^2}{s+1}} \left(\frac{s+1}{s+2}\right)^{\!\frac{s+2}{2}} \leq \sqrt{\frac{4\,\tau^2}{\E(s+1)}} \leq \sqrt{\frac{16\,\tau^2}{\E\lambda}} \leq\E^{-\frac{1}{2}}.$$ 
Therefore, 
\begin{align}
\sum_{s+1\geq \max(\frac{\lambda}{4},r+1)} \left(\frac{4\,\tau^2}{s+1}\right)^{\frac{s+1}{2}} &\leq \frac{1}{1-\E^{-\frac{1}{2}}} \left(\frac{16\,\tau^2}{\lambda}\right)^{\frac{r+1}{2}};\notag\\ 
\sum_{s+1\geq \max(\frac{\lambda}{4},r+1)} \dtv(\nu^{(s)},\nu^{(s+1)}) &\leq a_0\, \left(\frac{16\,\tau^2}{\lambda }\right)^{\frac{r+1}{2}}\tag{U0}\label{eq:U0} 
\end{align}
with 
$$a_0 = \frac{1}{2(1-\E^{-\frac{1}{2}})} \leq 1.2708.$$
If $r+1\geq \frac{\lambda}{4}$, then we are done. In the sequel, we suppose that $\frac{\lambda}{4} > r+1 \geq 2$, and we are going to evaluate $\dtv(\nu^{(s)},\nu^{(s+1)})$ when $r+1\leq s+1 < \frac{\lambda}{4}$. At the end, we shall add to these controls the upper bound \eqref{eq:U0} computed above.
\medskip

\subsection{Removal of the tails of the distributions}
For $s+1 < \frac{\lambda}{4}$, we start by rewriting the local difference between $\nu^{(s+1)}$ and $\nu^{(s)}$ as a Poisson--Charlier polynomial:
$$\nu^{(s+1)}(k) - \nu^{(s)}(k) = b_{s+1} \,\nu(k)\,\left(\sum_{l=0}^{\min(s+1,k)} (-1)^{s+1-l}\,\binom{s+1}{l}\,\frac{k!\,\lambda^{-l}}{(k-l)!}\right).$$
We refer to \cite[Remark 3.9]{CDMN20} for a proof of this formula, and to \cite[Section 2.8.1]{Sze39} for the general properties of the Poisson--Charlier orthogonal polynomials.
Since $\binom{s+1}{l} \leq \frac{(s+1)^l}{l!}$,
$$|\nu^{(s+1)}(k) - \nu^{(s)}(k)| \leq |b_{s+1}|   \left(1+\frac{s+1}{\lambda}\right)^{k} \,\nu(k)\leq |b_{s+1}|\,\left(\frac{5}{4}\right)^k\,\nu(k).$$
In order to estimate $\sum_{k\in \N}|\nu^{(s+1)}(k) - \nu^{(s)}(k)|$, we first remove the integers $k$ such that $k$ is too large. We set
$$\beta = \frac{4}{5}\left(1-\frac{1}{16\,\E}\right)\simeq 0.78160\ldots$$
  and $\alpha\simeq 1.73026\ldots$ such that $\alpha-\alpha\log \alpha = \beta$, and we remove the integers $k$ larger than $\frac{5}{4}\alpha\lambda$. We have:
\begin{align*}
\sum_{k> \frac{5}{4}\alpha\lambda} |\nu^{(s+1)}(k) - \nu^{(s)}(k)| &\leq |b_{s+1}|   \sum_{k>\frac{5}{4}\alpha\lambda} \left(\frac{5}{4}\right)^k\, \proba[\mathcal{P}(\lambda) = k] = |b_{s+1}|\,\E^{\frac{\lambda}{4}}   \sum_{k>\frac{5}{4}\alpha\lambda} \proba\!\left[\mathcal{P}\left(\frac{5}{4}\lambda\right) = k\right],
\end{align*}
and the right-hand side is proportional to the tail of a Poisson distribution with parameter $\frac{5}{4}\lambda$, so it can be estimated by using for instance the Chernov inequality:
\begin{align*}
\proba\!\left[\mathcal{P}\left(\frac{5}{4}\lambda\right) > \frac{5}{4}\alpha\lambda\right] &\leq \inf_{t \geq 0}\left(\E^{\frac{5}{4}\lambda(\E^{t}-1-\alpha t)} \right) = \E^{\frac{5}{4}(\beta-1)\lambda}.
\end{align*}
Thus, for $s+1<\frac{\lambda}{4}$,
$$\frac{1}{2}\sum_{k> \frac{5}{4}\alpha\lambda} |\nu^{(s+1)}(k) - \nu^{(s)}(k)| \leq \frac{|b_{s+1}|}{2}\,\E^{(\frac{5}{4}\beta-1)\lambda} \leq \frac{1}{2}\left(\frac{\tau^2}{\lambda(s+1)}\right)^{\!\frac{s+1}{2}}\,\lambda^{\frac{s+1}{2}}\E^{(\frac{5}{4}\beta-1)\lambda}.
$$
Notice that $\frac{5}{4}\beta-1 = -\frac{1}{16\,\E}$.
The function $\lambda \mapsto \lambda^{\frac{s+1}{2}}\E^{-\frac{\lambda}{16\,\E}}$ attains its maximum at $\lambda=8\,\E(s+1)$, so
$$\lambda^{\frac{s+1}{2}}\E^{(\frac{5}{4}\beta-1)\lambda} \leq (8(s+1))^{\frac{s+1}{2}}.$$ 
Therefore, we obtain the first upper bound:
\[
\frac{1}{2}\sum_{k> \frac{5}{4}\alpha\lambda} |\nu^{(s+1)}(k) - \nu^{(s)}(k)| \leq a_1\,\left(\frac{8\,\tau^2}{\lambda}\right)^{\!\frac{s+1}{2}} \tag{U1}\label{eq:U1}
\]
with $a_1=\frac{1}{2}$. For future reference, $\frac{5}{4}\alpha = \alpha' \simeq 2.16282\ldots$\medskip

\subsection{Removal of the tails of the Fourier integrals}
Suppose now that $k$ belongs to the interval $[0, \alpha'\lambda]$. By the Fourier inversion formula,
$$I_k^{(s)}=\nu^{(s+1)}(k) - \nu^{(s)}(k) = b_{s+1}\int_{(-\pi,\pi)} (\E^{\I \xi}-1)^{s+1}\, \E^{\lambda(\E^{\I \xi}-1) - \I k\xi} \,\frac{\!\DD{\xi}}{2\pi}.$$
 With $k = \lambda + x\lambda^{\frac{1}{2}}$ and $|x|=O(\lambda^{\frac{1}{2}})$, the idea is to approximate 
$(\E^{\I \xi}-1)^{s+1}$ by $(\I \xi)^{s+1}$ and $\lambda(\E^{\I\xi}-1) - \I k \xi$ by $-\frac{\lambda \xi^2}{2}- \I x (\lambda^{\frac{1}{2}}\xi)$, thereby obtaining an integral which can be computed explicitly and which yields a term proportional to the Hermite polynomial $H_{s+1}(x)$. Let us first split the integral $I_{k}^{(s)}$ in two parts $I_{k,1}^{(s)}$ and $I_{k,2}^{(s)}$, according to whether $|\xi|$ is larger or smaller than $\lambda^{-\frac{1}{3}}$.  For any $\xi \in (-\pi,\pi)$, 
$$\mathrm{Re}(\E^{\I \xi}-1) = \cos\xi - 1 \leq -\frac{2\xi^2}{\pi^2}, $$ 
so the integral $I_{k,1}^{(s)}$ corresponding to the outside of the interval $(-\lambda^{-\frac{1}{3}},\lambda^{-\frac{1}{3}})$ is smaller than
$$
 \frac{|b_{s+1}|}{\pi} \int_{\lambda^{-\frac{1}{3}}}^{\infty} \E^{-\frac{2\lambda\xi^2}{\pi^2}} \,\xi^{s+1} \DD{\xi} =  \frac{|b_{s+1}|}{2\pi}  \left(\frac{\pi^2}{2\lambda}\right)^{\!\frac{s}{2}+1} \,\Gamma\!\left(\frac{s}{2}+1,\frac{2\lambda^{\frac{1}{3}}}{\pi^2}\right),$$ 
 where $\Gamma(s,x) = \int_{x}^\infty u^{s-1}\,\E^{-u}\DD{u}$ denotes the incomplete $\Gamma$ function. Notice that for $t>0$,
$$
x^{t}\,\Gamma(s,x) = \int_{x}^\infty x^{t}\,u^{s-1}\,\E^{-u}\DD{u} \leq \int_{x}^\infty u^{s+t-1}\,\E^{-u}\DD{u} = \Gamma(s+t,x) \leq \Gamma(s+t).
$$
 With $\lambda > 4(s+1) \geq 8$, there are less than $\alpha'\lambda+1 \leq (\alpha' + \frac{1}{8})\,\lambda=\alpha''\lambda$ with $\alpha'' \simeq 2.28782\ldots$ integers $k$ in the interval $[0,\alpha'\lambda]$, so we get: 
 \begin{align}
\frac{1}{2}\sum_{k\leq \alpha'\lambda} I_{k,1}^{(s)} &\leq \frac{\pi^3}{16}\,\alpha''\,|b_{s+1}|  \left(\frac{\pi^2}{2\lambda}\right)^{\!\frac{s+1}{2}}\, \left(\frac{2\lambda^{\frac{1}{3}}}{\pi^2}\right)^{\!\frac{3}{2}}\,\Gamma\!\left(\frac{s}{2}+1,\frac{2\lambda^{\frac{1}{3}}}{\pi^2}\right)
 \notag\\ 
&\leq \frac{\pi^3}{16}\,\alpha''\,|b_{s+1}|\left(\frac{\pi^2}{2\lambda}\right)^{\!\frac{s+1}{2}}\,\Gamma\!\left(\frac{s+3}{2}+1\right) \notag \\ 
&\leq a_2\, (s+3)^{\frac{3}{2}}\,\left(\frac{\pi^2\tau^2}{4\E\lambda}\right)^{\!\frac{s+1}{2}}\tag{U2}\label{eq:U2}
\end{align}
with
$$
a_2 = \frac{\pi^\frac{7}{2}}{32}\,\alpha'' \leq 3.9292.
$$
We now focus on the part $I_{k,2}^{(s)}$ of the integral $I_k^{(s)}$ corresponding to the interval $(-\lambda^{-\frac{1}{3}},\lambda^{-\frac{1}{3}})$. First, let us notice that on this interval,
\begin{align*}
\lambda\left|\E^{\I \xi}-1-\I\xi+\frac{\xi^2}{2}\right| &\leq \frac{\lambda\,|\xi|^3}{6};\\ 
\left|\E^{\lambda(\E^{\I \xi}-1-\I\xi+\frac{\xi^2}{2})}-1\right| &\leq \frac{\E^{\frac{1}{6}}}{6}\,\lambda\,|\xi|^3
\end{align*}
by using the inequality $|\E^{z}-1| \leq |z|\,\E^{|z|}$. This leads to:
\begin{align}
&\frac{1}{2}\sum_{k \leq \alpha'\lambda} \left|I_{k,2}^{(s)} - b_{s+1}\int_{-\lambda^{-\frac{1}{3}}}^{\lambda^{-\frac{1}{3}}} (\E^{\I \xi}-1)^{s+1} \,\E^{-\frac{\lambda \xi^2}{2} - \I x (\lambda^{\frac{1}{2}}\xi)} \,\frac{\!\DD{\xi}}{2\pi}\right| \notag\\ 
&\leq \frac{\E^{\frac{1}{6}}}{6}\,\alpha'' \,|b_{s+1}|\, \lambda^{2} \int_{\R_+} \xi^{s+4}\,\E^{-\frac{\lambda \xi^2}{2}}\,\frac{\!\DD{\xi}}{2\pi} \notag \\ 
&\leq  a_3\, (s+3)^{\frac{3}{2}}\left(\frac{\tau^2}{\E\lambda}\right)^{\frac{s+1}{2}}\tag{U3}\label{eq:U3}
\end{align}
with 
$$a_3 = \frac{\E^{\frac{1}{6}}}{12\sqrt{\pi}} \,\alpha''\leq 0.1271.$$
Thus, we have proved so far:
\begin{proposition}
Under the hypotheses stated at the beginning of this section, for any index $s$ such that $2 \leq s+1 <\frac{\lambda}{4}$,
$\frac{1}{2}\sum_{k \in \N} |\nu^{(s+1)}(k)-\nu^{(s)}(k)|$ is smaller than
$$A_1+A_2+A_3 +\frac{1}{2} \,|b_{s+1}| \sum_{k\leq \alpha'\lambda}  \left|\underbrace{\int_{-\lambda^{-\frac{1}{3}}}^{\lambda^{-\frac{1}{3}}} (\E^{\I \xi}-1)^{s+1} \,\E^{-\frac{\lambda \xi^2}{2} - \I x (\lambda^{\frac{1}{2}}\xi)} \,\frac{\!\DD{\xi}}{2\pi}}_{I_{k,3}^{(s)}}\right|
$$
where $k=\lambda+x\lambda^{\frac{1}{2}}$ and $A_1$, $A_2$ and $A_3$ are respectively bounded from above by \eqref{eq:U1}, \eqref{eq:U2} and \eqref{eq:U3}.
 \end{proposition}
Thus, in the inequality
$$\dtv(\mu,\nu^{(r)}) \leq \frac{1}{2} \sum_{s \geq r} \sum_{k \in \N} |\nu^{(s+1)}(k) - \nu^{(s)}(k)|,$$
we have removed from the right-hand side the terms
\begin{itemize}
     \item with $s+1>\frac{\lambda}{4}$ (upper bound \eqref{eq:U0}),
     \item  with $k>\alpha'\lambda$ (upper bound \eqref{eq:U1}),
     \item and then with $\xi>\lambda^{-\frac{1}{3}}$ in the Fourier inversion formula for $\nu^{(s+1)}(k) - \nu^{(s)}(k)$ (upper bounds \eqref{eq:U2} and \eqref{eq:U3}).
 \end{itemize}   
 The next Section \ref{sec:hermite} is devoted to the analysis of the remaining terms $I_{k,3}^{(s)}$; the summation of all the estimates will then be performed in Section \ref{sec:riemann_sum}.
\bigskip

\section{Estimates of partial Hermite functions}\label{sec:hermite}
We keep the hypotheses stated at the very beginning of Section \ref{sec:fourier_inversion}, and which are those of our  Theorem \ref{main:A}. In order to control the integral $I_{k,3}^{(s)}$, we use the expansion in series
$$(\E^{\I \xi}-1)^{s+1} = \sum_{ n_1,\ldots, n_{s+1} \geq 1} \frac{(\I \xi)^{n_1+\cdots+n_{s+1}}}{n_1!\cdots n_{s+1}!}.$$
For $m \geq s+1$, set 
$$
C_{m,s} = \sum_{\substack{n_1,\ldots,n_{s+1} \geq 1 \\ n_1+\cdots+n_{s+1}=m}} \frac{1}{n_1!\cdots n_{s+1}!}  ,
$$
so that 
$(\E^{\I \xi}-1)^{s+1} = \sum_{m\geq s+1} C_{m,s}\,(\I \xi)^{m}$. By comparison with the sum over all non-negative integers $n_1, n_2,\ldots,n_{s+1}$, we see that  $C_{m,s} \leq \frac{(s+1)^m}{m!}$. On the other hand, 
$$I_{k,3}^{(s)} = \sum_{m \geq s+1} C_{m,s} \,\lambda^{-\frac{m+1}{2}}\,(-1)^m\,\int_{-\lambda^{\frac{1}{6}}}^{\lambda^{\frac{1}{6}}} (\I u)^m \,\E^{-\frac{u^2}{2} + \I x u} \,\frac{\!\DD{u}}{2\pi}.$$
We would like to replace the integrals over $(-\lambda^{\frac{1}{6}},\lambda^{\frac{1}{6}})$ by integrals over $\R$. Recall that for any $m \geq 0$, the Hermite polynomial $H_m(x)$ is given by:
$$\E^{-\frac{x^2}{2}}\,H_m(x) = (-1)^m\,\int_\R (\I u)^m \,\E^{-\frac{u^2}{2} + \I xu}\,\frac{\!\DD{u}}{\sqrt{2\pi}};$$
see \cite[Chapter V]{Sze39}.
The classical Cramér inequality states that $|H_m(x)|\leq \E^{\frac{x^2}{4}}\,\sqrt{m!}$ for any $x \in \R$ and any integer $m$. In the sequel, we prove a similar inequality when the integral is taken over an interval $[M,+\infty)$ instead of $\R$.
\medskip

\subsection{Partial Hermite functions as integrals of Hermite polynomials}
Fix $M>0$, and set $f_m(x) = (-1)^m\,\int_M^\infty (\I u)^m \,\E^{-\frac{u^2}{2}+\I xu}\,\frac{\!\DD{u}}{\sqrt{2\pi}}$. We have $f_m(x) = (-1)^m\,\frac{\partial^m}{\partial x^m} f_0(x)$, and
$$f_0 = \mathcal{F}\left(S_M(u)\,\E^{-\frac{u^2}{2}}\right),$$
where $S_M$ is the Heaviside step function and $(\mathcal{F}f)(x) = (2\pi)^{-1/2}\int_\R \E^{\I x u}\,f(u)\,\DD{u}$ is the Fourier transform. The inverse Fourier transform is given by $(\mathcal{F}^{-1}g)(u) =  (2\pi)^{-1/2}  \int_\R \E^{-\I x u}\,g(x)\,\DD{x}$. The Fourier transform leaves the function $$N(u)= \E^{-\frac{u^2}{2}}$$ invariant. Since $\mathcal{F}^{-1}(f*g) = \sqrt{2\pi}\,(\mathcal{F}^{-1}f)\,(\mathcal{F}^{-1}g)$, if $S_M(u) = (\mathcal{F}^{-1}K_M)(u)$, then 
$$f_0=\mathcal{F}((\mathcal{F}^{-1}K_M)(u) \,(\mathcal{F}^{-1}N)(u)) = \frac{1}{\sqrt{2\pi}}\,(K_M *N).$$
However, the Fourier transform $K_0$ of $S_0$ is given by
$$K_0(x) = \frac{\I}{\sqrt{2\pi}\,x} + \sqrt{\frac{\pi}{2}}\,\delta_0(x);$$
see for instance \cite[Chapter 9]{Dav02}. Then, $K_M(x) = \E^{\I M x}\,K_0(x)$ and we get:
$$f_0(x) = \frac{\I}{2\pi} \left( \int_{\R} \E^{\I M y-\frac{(x-y)^2}{2}}\,\frac{\!\DD{y}}{y} \right) + \frac{1}{2}\,\E^{-\frac{x^2}{2}}.$$ 
In this formula, the first term involves a principal value, and it can be replaced by:
\begin{align*}
\int_{\R} \E^{\I M y-\frac{(x-y)^2}{2}}\,\frac{\!\DD{y}}{y} &= 2 \E^{-\frac{x^2}{2}} \int_0^\infty \E^{-\frac{y^2}{2}}  \,\frac{\sinh((x+\I M)y)}{y} \DD{y} \\ 
&=2 \E^{-\frac{x^2}{2}} \sum_{n=0}^\infty \frac{(x+\I M)^{2n+1}}{(2n+1)!} \int_0^\infty \E^{-\frac{y^2}{2}} \,y^{2n}\DD{y} \\ 
&=\sqrt{2\pi}(x+\I M) \,\E^{-\frac{x^2}{2}} \sum_{n=0}^\infty \frac{1}{(2n+1)\, (n!)}\left(\frac{(x+\I M)^2}{2}\right)^n \\ 
&=\sqrt{2\pi}(x+\I M) \int_{t=0}^1 \E^{-\frac{x^2}{2}} \,\E^{\frac{t^2(x+\I M)^2}{2}}\DD{t}.
\end{align*}
We have $$\frac{x^2}{2}-\frac{t^2(x+\I M)^2}{2} = \frac{1}{2}\left(\sqrt{1-t^2}\,x - \frac{\I t^2 M}{\sqrt{1-t^2}}\right)^2 + \frac{t^2M^2}{2(1-t^2)},$$ 
and on the other hand, $\E^{-\frac{x^2}{2}}\,H_m(x) = (-1)^m\,\frac{\partial^m}{\partial x^m}(\E^{-\frac{x^2}{2}})$. So, for any $m \geq 0$,
\begin{align*}
(-1)^m\,\frac{\partial^m}{\partial x^m}\left(\E^{-\frac{x^2}{2}} \,\E^{\frac{t^2(x+\I M)^2}{2}}\right) = \E^{-\frac{x^2}{2}} \,\E^{\frac{t^2(x+\I M)^2}{2}}\,(1-t^2)^{\frac{m}{2}}\,H_m\!\left(\sqrt{1-t^2}\,x- \frac{\I t^2 M}{\sqrt{1-t^2}}\right)
\end{align*}
and we get the semi-explicit formula:
\begin{align*}
f_m(x) &= \frac{\I(x+\I M)}{\sqrt{2\pi}} \int_{t=0}^1 \E^{-\frac{x^2}{2}} \,\E^{\frac{t^2(x+\I M)^2}{2}}\,(1-t^2)^{\frac{m}{2}}\,H_m\!\left(\sqrt{1-t^2}\,x- \frac{\I t^2 M}{\sqrt{1-t^2}}\right) \DD{t} \\ 
&- \frac{\I}{\sqrt{2\pi}} \int_{t=0}^1 \E^{-\frac{x^2}{2}} \,\E^{\frac{t^2(x+\I M)^2}{2}}\,(1-t^2)^{\frac{m-1}{2}}\,H_{m-1}\!\left(\sqrt{1-t^2}\,x - \frac{\I t^2 M}{\sqrt{1-t^2}}\right) \DD{t} +\frac{1}{2}\,\E^{-\frac{x^2}{2}}\,H_m(x).
\end{align*}
\medskip

\subsection{Control of the integrands}
In the integral formula for $f_m(x)$, in order to remove the factor $\frac{1}{\sqrt{1-t^2}}$ which is divergent when $t$ goes to $1$, we shall use the following identity:
$$(1-t^2)^{\frac{m}{2}}\,H_m\!\left(\sqrt{1-t^2}\,x- \frac{\I t^2 M}{\sqrt{1-t^2}}\right) = \sum_{l=0}^{\lfloor \frac{m}{2}\rfloor} t^{2l}\,\binom{m}{2l}\,\frac{(2l)!}{2^l\,l!}\,H_{m-2l}((1-t^2)\,x-\I t^2 M).$$
This is a particular case of:
\begin{lemma}[Multiplication theorem for Hermite polynomials]
For any $m \geq 0$,
$$H_m(ax) = \sum_{l=0}^{\lfloor \frac{m}{2}\rfloor} a^{m-2l}\,(a^2-1)^l\,\binom{m}{2l}\,\frac{(2l)!}{2^l\,l!}\,H_{m-2l}(x).$$
\end{lemma}

\begin{proof}
The Hermite polynomials satisfy the recurrence equation $$H_{m+1}(x)=x\,H_m(x)-m\,H_{m-1}(x);$$ see \cite[Equation (5.5.8)]{Sze39}. This implies the explicit expression:
$$H_m(x) = \sum_{l=0}^{\lfloor \frac{m}{2}\rfloor} \frac{(-1)^l\,m!}{2^l\,(m-2l)!\,l!}\,x^{m-2l}.$$
The inverse of this formula is:
$$x^m = \sum_{l=0}^{\lfloor \frac{m}{2} \rfloor} \frac{m!}{2^l\,(m-2l)!\,l!}\,H_{m-2l}(x).$$
Therefore,
\begin{align*}
H_m(ax) &= \sum_{2j \leq m} \frac{(-1)^j\,m!}{2^j\,(m-2j)!\,j!}\,a^{m-2j}\,x^{m-2j} \\ 
&= \sum_{j,k\,|\,2(j+k) \leq m} \frac{(-1)^j\,m!}{2^{j+k}\,(m-2j-2k)!\,j!\,k!}\,a^{m-2j}\, H_{m-2j-2k}(x) \\ 
&= \sum_{l=0}^{\lfloor \frac{m}{2}\rfloor} \left(\sum_{j=0}^l \binom{l}{j}\,(-1)^j  a^{m-2j}\right)\frac{m!}{2^{l}\,(m-2l)!\,l!}\,H_{m-2l}(x) \\
&= \sum_{l=0}^{\lfloor \frac{m}{2}\rfloor} a^{m-2l}\,(a^2-1)^l \frac{m!}{2^{l}\,(m-2l)!\,l!}\,H_{m-2l}(x),
\end{align*}
which is the claimed formula up to a rewriting of the binomial coefficient.
\end{proof}
\medskip

The Hermite polynomial $H_n(z)$ is given  by the contour integral
$$H_n(z) = \frac{n!}{2\I \pi}\oint \E^{zu - \frac{u^2}{2}}\frac{\!\DD{u}}{u^{n+1}};$$
see \cite[Equation (5.5.12)]{Sze39}.
As a consequence, for any $z \in \C$, setting $|z| = \sqrt{n}\,y$, we get
\begin{align*}
\frac{|H_n(z) |}{n!} &\leq \min_{R>0}\left(\E^{R|z| + \frac{R^2}{2} - n\log R}\right) \\ 
&\leq \E^{\frac{n-n\log n}{2}}\,\exp\left(n\left(y\,\frac{\sqrt{y^2+4}-y}{4} - \log\left(\frac{\sqrt{y^2+4}-y}{2}\right)\right)\right).
\end{align*}
The function $$y \mapsto y\,\frac{\sqrt{y^2+4}-y}{4} - \log(\frac{\sqrt{y^2+4}-y}{2})$$ is increasing from $\R_+$ to $\R_+$, and it behaves asymptotically as $\log y$. Therefore, it is smaller than $\frac{y^2}{4}+c$ for some constant $c$, which can for instance be taken equal to $\frac{\log 3}{2}$. By using the Stirling estimates $\sqrt{2\pi n}\leq \frac{n!}{\E^{n\log n -n}} \leq \E\sqrt{n}$, we thus get a complex version of the Cramér inequality for Hermite functions:
$$\forall n \geq1,\,\,\,\forall z \in \C,\,\,\,|H_n(z)| \leq \E^{\frac{|z|^2}{4}} \,3^{\frac{n}{2}} \sqrt{n!} \, \E^{\frac{1}{2}}\,n^{\frac{1}{4}}.$$
So,
\begin{align*}
&\left|\E^{-\frac{x^2}{2}} \,\E^{\frac{t^2(x+\I M)^2}{2}}\,(1-t^2)^{\frac{m}{2}}\,H_m\!\left(\sqrt{1-t^2}\,x- \frac{\I t^2 M}{\sqrt{1-t^2}}\right)\right| \\ 
&\leq 3^{\frac{m}{2}}\,\E^{-\frac{(1-t^4)x^2+t^2(2-t^2)M^2}{4}+\frac{1}{2}} \left(\sum_{l=0}^{\lfloor \frac{m}{2}\rfloor} \frac{m!}{l! \sqrt{(m-2l)!}}\, (m-2l)^{\frac{1}{4}} \left(\frac{t^{2}}{6}\right)^{\!l}\right).
\end{align*}
When $m=2n$ is even, $\sqrt{(m-2l)!} = \sqrt{(2(n-l))!} \geq 2^{n-l} \frac{(8\pi)^{\frac{1}{4}}}{\E}\, (n-l)!\,(m-2l)^{-\frac{1}{4}}$, so the sum is smaller than
$$
D_m=\frac{\E}{(8\pi)^{\frac{1}{4}}} \, m^{\frac{1}{2}}\,\frac{m!}{\left(\frac{m}{2}\right)!}\left(\frac{2}{3}\right)^{\!\frac{m}{2}} \leq 2 \, m^{\frac{m+1}{2}} \left(\frac{4}{3\E}\right)^{\!\frac{m}{2}}. %in fact with e^2/(2pi)^(3/4) instead of 2.
$$
When $m=2n+1$ is odd, $\sqrt{(m-2l)!} = \sqrt{(2(n-l)+1)!} \geq 2^{n-l} \frac{(8\pi)^{\frac{1}{4}}}{\E}\, (n-l)!\, (m-2l)^{\frac{1}{4}}$, so the sum is in this case smaller than
$$
D_m=\frac{\E}{(8\pi)^{\frac{1}{4}}} \,\frac{m!}{(\frac{m-1}{2})!} \left(\frac{2}{3}\right)^{\!\frac{m-1}{2}} \leq 2 \,m^{\frac{m+1}{2}}   \left(\frac{4}{3\E}\right)^{\!\frac{m}{2}}.
$$
 Set $\Theta(x,M) = \sqrt{x^2+M^2}\,\int_{t=0}^1 \E^{-\frac{(1-t^4)x^2+t^2(2-t^2)M^2}{4}}\DD{t}$, and $g_m(x) = \frac{1}{2}\,\E^{-\frac{x^2}{2}}\,H_m(x)-f_m(x)$. We are interested in the case where $M = \lambda^{\frac{1}{6}}$ and $|x|=O( \lambda^{\frac{1}{2}})$. The previous calculations prove that for $m \geq 2$ and $M \geq \sqrt{2}$, 
\begin{align*}
|g_m(x)| \leq \frac{5\E}{4\sqrt{2\pi}} \,\left(\frac{4m}{\E}\right)^{\!\frac{m+1}{2}}\,\Theta(x,M).
\end{align*}
Since $H_m(x) = (-1)^m\,H_m(-x)$, we have:
\begin{align*}
&(-1)^m\int_{-M}^M (\I u)^m\,\E^{-\frac{u^2}{2}+\I xu}\,\frac{\!\DD{u}}{\sqrt{2\pi}} \\ 
&= \E^{-\frac{x^2}{2}}\,H_m(x) - (-1)^m \int_M^\infty (\I u)^m\,\E^{-\frac{u^2}{2}+\I xu}\,\frac{\!\DD{u}}{\sqrt{2\pi}} - (-1)^m \int_{-\infty}^{-M} (\I u)^m\,\E^{-\frac{u^2}{2}+\I xu}\,\frac{\!\DD{u}}{\sqrt{2\pi}} \\ 
&= \E^{-\frac{x^2}{2}}\,H_m(x) - f_m(x) - (-1)^m f_m(-x) =g_m(x) + (-1)^m\,g_m(-x).
\end{align*}
Therefore, for $s +1 < \frac{\lambda}{4}$,
\begin{align*}
I_{k,3}^{(s)}&\leq  \frac{5\E}{2\pi} \left(\sum_{m \geq s+1} C_{m,s}\,\left(\frac{4m}{\E\lambda}\right)^{\!\frac{m+1}{2}}\right)\Theta(x,\lambda^{\frac{1}{6}})\\ 
&\leq \frac{5\E^{\frac{1}{2}}}{\sqrt{2\pi^3}}\,\lambda^{-\frac{1}{2}}\underbrace{\left(\sum_{m \geq s+1} \left(\frac{4\E(s+1)^2}{\lambda m}\right)^{\frac{m}{2}}\right)}_{E_{s}}\Theta(x,\lambda^{\frac{1}{6}}).
\end{align*}
\medskip

\subsection{Control of the series}
Let us explain how to control the series $E_{s}$. If we consider $$\left(\frac{4\E(s+1)^2}{\lambda m}\right)^{\frac{m}{2}}$$ as a function of $m>0$, then it is increasing until $m=\frac{4(s+1)^2}{\lambda}$, and then decreasing. Since we assume $\lambda>4(s+1)$, the first term of the series is therefore the largest one. We bound the $(s+1)$ first terms of the series by $(\frac{4\E(s+1)}{\lambda})^{\frac{s+1}{2}}$, and for $m \geq 2(s+1)$, we have:
$$\left(\frac{4\E(s+1)^2}{\lambda (m+1)}\right)^{\frac{m+1}{2}} \leq \sqrt{\frac{4(s+1)^2}{\lambda m}}\left(\frac{4\E(s+1)^2}{\lambda m}\right)^{\!\frac{m}{2}}\leq \sqrt{\frac{1}{2}}\left(\frac{4\E(s+1)^2}{\lambda m}\right)^{\!\frac{m}{2}}.$$
Therefore,
\begin{align*}
E_{s} &\leq (s+1)\left(\frac{4\E(s+1)}{\lambda}\right)^{\frac{s+1}{2}} + \left(\sum_{n=0}^\infty 2^{-\frac{n}{2}} \right) \left(\frac{2\E(s+1)}{\lambda}\right)^{s+1} \\ 
&\leq \left(1 + \frac{\E}{8(1-2^{-\frac{1}{2}})} \right)(s+1)\left(\frac{4\E(s+1)}{\lambda}\right)^{\frac{s+1}{2}} .
\end{align*}
We can summarise the estimates computed in this section:
\begin{proposition}\label{prop:theta}
Under the hypotheses stated at the very beginning of Section \ref{sec:fourier_inversion}, for any $s$ such that $2 \leq s+1 < \frac{\lambda}{4}$, if $k=\lambda + x\lambda^{\frac{1}{2}}$, then
$$\frac{1}{2}\,|b_{s+1}|\,I_{k,3}^{(s)} \leq a_4\, \lambda^{-\frac{1}{2}}\,(s+1) \left(\frac{4\E\tau^2}{\lambda}\right)^{\!\frac{s+1}{2}} \,\Theta(x,\lambda^{\frac{1}{6}}),$$ 
with $\Theta(x,M) = \sqrt{x^2+M^2}\,\int_{t=0}^1 \E^{-\frac{(1-t^4)x^2+t^2(2-t^2)M^2}{4}}\DD{t}$, and 
$$a_4=\frac{5\E^{\frac{1}{2}}}{(2\pi)^{\frac{3}{2}}} \left(1 + \frac{\E}{8(1-2^{-\frac{1}{2}})}\right)\leq 1.1307.$$
\end{proposition}
\medskip

\section{Riemann summation of the estimates}\label{sec:riemann_sum}
We now gather all the estimates previously computed in order to prove Theorem \ref{main:A}. Setting $\eps=\frac{4\,\tau}{\sqrt{\lambda}}<1$, and we start by comparing $\eps^{r+1}$ to the sums of the upper bounds \eqref{eq:U1}, \eqref{eq:U2} and \eqref{eq:U3} over indices $s$ such that $r+1\leq s+1 < \frac{\lambda}{4}$.

\begin{itemize}
    \item The sum of the upper bounds \eqref{eq:U1} is smaller than
    $$\frac{1}{4(1-\frac{1}{\sqrt{2}})}\,\eps^{r+1}\leq 0.8536\,\eps^{r+1}.$$ 
    
    \item Set $\eps_1 = \frac{\pi\tau}{2\sqrt{\E\lambda}}$. The sum of the upper bounds \eqref{eq:U2} is smaller than
    $$a_2 \sum_{s \geq r} (s+3)^{\frac{3}{2}}\,(\eps_1)^{s+1} \leq a_2\,\frac{(r+3)^{\frac{1}{2}}}{r+2}\sum_{s \geq r} (s+3)(s+2)\,(\eps_1)^{s+1}$$
    with $\eps_1\leq \frac{\pi}{8\sqrt{\E}}$. The series is equal to
    \begin{align*}
    &(r+3)(r+2)\,(\eps_1)^{r+1}\left(\frac{1}{1-\eps_1}+ \frac{2\eps_1}{(r+2)(1-\eps_1)^2}+ \frac{2(\eps_1)^2}{(r+2)(r+3)(1-\eps_1)^3}\right) \\ 
    &\leq 1.6077\,(r+3)(r+2)\,(\eps_1)^{r+1}.
    \end{align*}
    Thus, the contribution of the upper bounds \eqref{eq:U2} is smaller than 
    $$6.3167\,(r+3)^{\frac{3}{2}}\,(\eps_1)^{r+1}\leq 6.3167\,(r+3)^{\frac{3}{2}}\,\left(\frac{\pi}{8\sqrt{\E}}\right)^{r+1}\, \eps^{r+1}\leq 2.8669\,\eps^{r+1}.$$
    
    \item Similarly, with $\eps_2 = \frac{\tau}{\sqrt{\E\lambda}}<\frac{1}{4\sqrt{\E}}$, the sum of the upper bounds \eqref{eq:U3} is smaller than
    \begin{align*}
    &a_3\,\frac{(r+3)^{\frac{1}{2}}}{r+2}\sum_{s \geq r} (s+3)(s+2)\,(\eps_2)^{s+1} \leq 0.1685\,(r+3)^{\frac{3}{2}} (\eps_2)^{r+1} \leq 0.031\,\eps^{r+1}.
    \end{align*}
\end{itemize}
Therefore, taking also into account \eqref{eq:U0} and using Proposition \ref{prop:theta}, we see that for any $r \geq 1$,
$$ \dtv(\mu,\nu^{(r)}) \leq 5.0223\,\eps^{r+1} + a_4\left(\sum_{s \geq r} \,(s+1)\, \left(\frac{4\E\tau^2}{\lambda}\right)^{\!\frac{s+1}{2}}\right) \left(\sum_{k\leq \alpha'\lambda}\lambda^{-\frac{1}{2}} \,\Theta(x,\lambda^{\frac{1}{6}})\right).$$
The sum over indices $s$ is easy to compare with $\eps^{r+1}$: with $\eps_3=\frac{2\sqrt{\E}\tau}{\sqrt{\lambda}}<\frac{\sqrt{\E}}{2}$, it is equal to 
$$(r+1)(\eps_3)^{r+1}\left(\frac{1}{1-\eps_3}+\frac{\eps_3}{(r+1)(1-\eps_3)^2}\right)\leq 25.898\,\eps^{r+1}. $$
So,
$$ \dtv(\mu,\nu^{(r)}) \leq 5.0223\,\eps^{r+1} + 29.2811\,\eps^{r+1} \left(\sum_{k\leq \alpha'\lambda}\lambda^{-\frac{1}{2}} \,\Theta(x,\lambda^{\frac{1}{6}})\right)$$
and in the remaining sum, we can assume $\lambda \geq 8$, since otherwise the upper bound \eqref{eq:U0} suffices. 
\medskip

In order to control a sum $\sum_{k \leq \alpha'\lambda}\Theta(x,M)$, we split the integral $\int_{t=0}^1 \E^{-\frac{(1-t^4)x^2 + t^2(2-t^2)M^2}{4}} \DD{t}$ according to whether $t^2$ is smaller or larger than $\frac{1}{2}$. If $t^2<\frac{1}{2}$, then $$\frac{(1-t^4)x^2 + t^2(2-t^2)M^2}{4} \geq \frac{3x^2}{16}+\frac{3t^2M^2}{8},$$ whereas if $t^2>\frac{1}{2}$, then $$\frac{(1-t^4)x^2 + t^2(2-t^2)M^2}{4} \geq  \frac{3(1+\frac{1}{\sqrt{2}})(1-t)x^2}{8}+\frac{3M^2}{16} .$$
Therefore,
 $$\Theta(x,\lambda^{\frac{1}{6}}) \leq (|x|+\lambda^{\frac{1}{6}})\left(\sqrt{\frac{2\pi}{3}}\,\frac{1}{\lambda^{\frac{1}{6}}}\,\E^{-\frac{3x^2}{16}}+ \left(1-\frac{1}{\sqrt{2}}\right)\E^{-\frac{3\lambda^{\frac{1}{3}}}{16}}\,\frac{1-\E^{-\frac{3x^2}{16}}}{\frac{3x^2}{16}}\right).$$
 We want to take the Riemann sum of the values $\lambda^{-\frac{1}{2}}\,\Theta(x,\lambda^{\frac{1}{6}})$ where $x$ runs over the set of real numbers 
 $$\left\{\frac{k-\lambda}{\lambda^{\frac{1}{2}}},\,\,k \in \lle0,\lfloor \alpha'\lambda\rfloor \rre\right\}.$$
 Let us remove the two values where $k=\lfloor \lambda \rfloor$ and $k = \lfloor \lambda \rfloor + 1$ ; they correspond to values of $|x|$ smaller than $\lambda^{-\frac{1}{2}}$, and they yield a contribution smaller than 
 $4.8086\,\lambda^{-\frac{1}{2}} $. The other values of $x$ can be considered as middle points of intervals $[a,b]$ with $b-a=\lambda^{-\frac{1}{2}}$, these intervals being included respectively in $[-\lambda^{\frac{1}{2}}-\frac{1}{2}\,\lambda^{-\frac{1}{2}},0]$ and $[0,(\alpha'-1)\lambda^{\frac{1}{2}}+\frac{1}{2}\,\lambda^{-\frac{1}{2}}]$.
For any twice-differentiable function $\theta : \R \to \R$ and any interval $[a,b]$, recall that
$$\left|\left(\int_a^b \theta(x)\DD{x} \right)-(b-a)\,\theta\!\left(\frac{a+b}{2}\right)\right| \leq \frac{\|\theta''\|_{\infty}(b-a)^3}{24};$$
see for instance the estimates at the end of \cite[Section 6.3]{Zor04}. We shall use this upper bound with the following functions $\theta$ on $\R_+$ :
\begin{center}
\begin{tabular}{|c|c|}
\hline $\theta$ & upper bound on $\|\theta''\|_\infty$ \\
\hline $x\,\E^{-\frac{3x^2}{16}}$ & $0.85$ \\ 
\hline $\E^{-\frac{3x^2}{16}}$ & $\frac{3}{8}$\\ 
\hline  $\frac{1-\E^{-\frac{3x^2}{16}}}{\frac{3x}{16}}$  & $0.54$\\ 
\hline $\frac{1-\E^{-\frac{3x^2}{16}}}{\frac{3x^2}{16}}$ & $\frac{3}{16}$\\ 
\hline 
\end{tabular}
\end{center}
We obtain:
\begin{align*}
\sum_{\substack{k \leq \alpha'\lambda \\ k \neq \lfloor \lambda \rfloor, \lfloor \lambda \rfloor+1}} \lambda^{-\frac{1}{2}}\,\Theta(x,\lambda^{\frac{1}{6}})
 &\leq I_1\,\lambda^{-\frac{1}{6}} + I_2 + \E^{-\frac{3\lambda^{\frac{1}{3}}}{6}}\,I_3 + I_4 + 0.1421\,\lambda^{-\frac{1}{2}},
\end{align*}
with \begin{align*}
I_1 &= \sqrt{\frac{2\pi}{3}} \int_\R |x|\,\E^{-\frac{3x^2}{16}}\DD{x} =\frac{16}{3}\sqrt{\frac{2\pi}{3}} \leq 7.7185;\\ 
I_2 &= \sqrt{\frac{2\pi}{3}} \int_\R \E^{-\frac{3x^2}{16}}\DD{x} = \frac{4\pi\sqrt{2}}{3}\leq 5.9239;\\ 
I_3&= \left(1-\frac{1}{\sqrt{2}}\right) \int_{-\lambda^{\frac{1}{2}}-\frac{1}{2}\,\lambda^{-\frac{1}{2}}}^{(\alpha'-1)\lambda^{\frac{1}{2}}+\frac{1}{2}\,\lambda^{-\frac{1}{2}}} \frac{1-\E^{-\frac{3x^2}{16}}}{\frac{3|x|}{16}}\,\DD{x} \leq1.87 \int_{-\lambda^{\frac{1}{2}}-\frac{1}{2}\,\lambda^{-\frac{1}{2}}}^{(\alpha'-1)\lambda^{\frac{1}{2}}+\frac{1}{2}\,\lambda^{-\frac{1}{2}}} \frac{1}{1+|x|}\,\DD{x};\\ 
I_4&= \left(1-\frac{1}{\sqrt{2}}\right) \int_{\R} \frac{1-\E^{-\frac{3x^2}{16}}}{\frac{3x^2}{16}}\,\DD{x} = \frac{4\sqrt{3\pi}(2-\sqrt{2})}{3} \leq 2.3979.
\end{align*}
The integral $I_3$ yields logarithms which are compensated by the term $\E^{-\frac{3\lambda^{1/3}}{16}}$, and it is then easy to check that the worst case is when $\lambda=8$; in this case,
$$\sum_{k\leq \alpha'\lambda}\lambda^{-\frac{1}{2}} \,\Theta(x,\lambda^{\frac{1}{6}}) \leq 19.2366.$$
Combining this estimate with the previous calculations, we finally obtain 
$$\dtv(\mu,\nu^{(r)})\leq  570\,\eps^{r+1}, $$
and this ends the proof of Theorem \ref{main:A}. 
\bigskip

\section{Hankel contours and unconditional upper bounds}\label{sec:hankel}
The purpose of this last section is how explain how to compute for the three examples from Section 2.2 the upper bound $\eps_n$ involved in our main Theorem \ref{main:C}. These computations rely on standard arguments from complex analysis, but as far as we know the estimates that we obtain have never been written with an explicit remainder $\eps_n$.\medskip

We consider a sequence $(X_n)_{n \in \N}$ of integer-valued random variables, and a sequence $(a_n)_{n\in \N}$ of real numbers, such that the double generating series
$$F(z,w) = \sum_{n=0}^\infty a_n\,z^n\,\esper[w^{X_n}]$$
has the following properties:
\begin{enumerate}[label=(H\arabic*)]
    \item\label{item:hankel_1} One can find a domain $\Delta(m,M,\phi) \times D(0,\rho)$ with $\rho>1$,
    such that $F(z,w)$ extends to a biholomorphic function on this domain.

    \item\label{item:hankel_2} There exist two constants $\theta>0$ and $T>0$ and a holomorphic function $L(w)$ on $D(0,\rho)$ such that 
    $$ F(z,w) = \left(\frac{1}{1-\frac{z}{m}}\right)^{\!\theta w} \,\exp(L(w)+O(|z-m|))$$
    on the domain $ \Delta(m,M,\phi) \times D(0,\rho)$, with an implied constant $T$ in the $O(\cdot)$.
\end{enumerate}
The numbers of cycles $(C_n)_{n \in \N}$ of random permutations chosen according to the probability measures $\proba_{n,(\theta_k)_{k \geq 1}}$ and the numbers of irreducible divisors $(D_n)_{n \in \N}$ of random polynomials in $\For_{q\geq 3}[X]$ satisfy the hypotheses \ref{item:hankel_1} and \ref{item:hankel_2} above. Indeed, with $a_n=\frakh_n(\Theta)$, $\rho=2$ and $L(w)=Kw$, under the assumptions of Proposition \ref{prop:random_permutations}, we have for $(C_n)_{n\in\N}$
$$F_C(z,w) = \sum_{n=0}^\infty z^n\,\frakh_n(w\Theta) = \exp(w\,\mathfrak{P}(\Theta,z))=\left(\frac{1}{1-\frac{z}{m}}\right)^{\!\theta w} \,\exp(Kw+O(2|z-m|)).$$
The terms $\theta$, $m$, $M$, $K$ and the implied constant in the $O(\cdot)$ above are the same as in the statement of Proposition \ref{prop:random_permutations} for the generating series $\mathfrak{P}(\Theta,z)$. For $(D_n)_{n\in\N}$ and $a_n=q^n$, we have
$$
F_D(z,w) = \exp(\sum_{k=1}^\infty \frac{\mathfrak{I}_q(z^k)\,(1-(1-w)^k)}{k})$$
with 
\begin{align*}
\mathfrak{I}_q(y) &= \sum_{n=1}^\infty I_q(n)\,y^n = \sum_{n=1}^\infty \sum_{d\mathrel{|}n} \mu(d)\,q^{\frac{n}{d}}\, \frac{y^n}{n} \\ 
&=\sum_{d=1}^\infty \frac{\mu(d)}{d}\, \log(\frac{1}{1-qy^d}) = \log(\frac{1}{1-qy}) + \mathfrak{R}_q(y),
\end{align*}
where $\mathfrak{R}_q(y) = \sum_{d=2}^\infty \frac{\mu(d)}{d}\, \log(\frac{1}{1-qy^d})$, which is convergent on the disk $D(0,q^{-\frac{1}{2}})$. Therefore,
$$F_D(z,w) = \left(\frac{1}{1-qz}\right)^{\!w} \exp(\mathfrak{R}_q(z)\,w +\sum_{k=2}^\infty \frac{\mathfrak{I}_q(z^k)\,(1-(1-w)^k)}{k}).$$
Let us remark right away that $F_D(z,1) = \frac{1}{1-qz}$, and therefore that $\sum_{k=2}^\infty \frac{\mathfrak{I}_q(z^k)}{k} = -\,\mathfrak{R}_q(z)$. As a consequence, we can rewrite the double generating series as:
$$F_D(z,w) = \left(\frac{1}{1-qz}\right)^{\!w} \exp(\mathfrak{R}_q(z)\,(w-1)+\sum_{k=2}^\infty \frac{\mathfrak{I}_q(z^k)\,(-1)^{k-1}\,(w-1)^k}{k}).$$
Therefore, the hypotheses \ref{item:hankel_1} and \ref{item:hankel_2} hold with $\rho=\frac{5}{4}$, $\theta=1$, $m=q^{-1}$, $M=q^{-\frac{3}{4}}$, 
\begin{align*}
L(w)&= R_q\,(w-1) +\sum_{k=2}^\infty \frac{\mathfrak{I}_q(q^{-k})\,(-1)^{k-1}\,(w-1)^k}{k} ; \\
T&=\frac{9}{4}\, \frac{q^{\frac{3}{2}}}{(q^{\frac{1}{2}}-1)(q^{\frac{3}{4}}-1)}+\frac{81}{16}\,\frac{q^{3}}{(q^{\frac{1}{2}}-1)(q^{\frac{3}{2}}-1)(q^{\frac{3}{4}}-\frac{9}{4})} .
\end{align*}
Indeed, for $|z| < q^{-\frac{3}{4}}$,
\begin{align*}
|\mathfrak{R}_q(z)-\mathfrak{R}_q(q^{-1})| &\leq \sum_{d=2}^\infty \frac{1}{d} \left|\log(\frac{1}{1-qz^d}) - \log(\frac{1}{1-q^{1-d}})\right| \\ 
&\leq \sum_{d=2}^\infty \frac{1}{d(1-q^{1-\frac{3d}{4}})} \,|qz^d-q^{1-d}|\leq \sum_{d=2}^\infty  \frac{q}{(1-q^{-\frac{1}{2}})} \, q^{-\frac{3(d-1)}{4}}\,|z-q^{-1}| \\ 
&\leq \frac{q^{\frac{3}{2}}}{(q^{\frac{1}{2}}-1)(q^{\frac{3}{4}}-1)}\,|z-q^{-1}|,
\end{align*}
and similarly, 
$$|\mathfrak{I}_q(z^k) - \mathfrak{I}_q(q^{-k})| \leq \frac{k\,q^{3}}{(q^{\frac{1}{2}}-1)(q^{\frac{3}{2}}-1)}\,q^{-\frac{3(k-1)}{4}}\,|z-q^{-1}| $$
for $k \geq 2$, so
\begin{align*}
\left|\sum_{k=2}^\infty \frac{(\mathfrak{I}_q(z^k)-\mathfrak{I}_q(q^{-k}))\,(1-w)^k}{k} \right| &\leq \frac{9}{4}\,\frac{q^{3}}{(q^{\frac{1}{2}}-1)(q^{\frac{3}{2}}-1)} \left(\sum_{k=2}^\infty \left(\frac{9}{4}\,q^{-\frac{3}{4}}\right)^{k-1}\right)|z-q^{-1}| \\ 
&\leq \frac{81}{16}\,\frac{q^{3}}{(q^{\frac{1}{2}}-1)(q^{\frac{3}{2}}-1)(q^{\frac{3}{4}}-\frac{9}{4})} \,|z-q^{-1}|
\end{align*}
since $|w-1| \leq 1+\rho=\frac{9}{4}$. The angle $\phi$ can here be chosen arbitrary small, since the only singularity on $D(0,M) \times D(0,\rho)$ of the double generating series $F_D(z,w)$ is at $z=m$.
\medskip

Following \cite{FO90}, under the assumptions \ref{item:hankel_1} and \ref{item:hankel_2}, let us compute $f_n(w)=a_n\,\esper[w^{X_n}]$ by using the Cauchy integral formula with respect to the following Hankel contour:
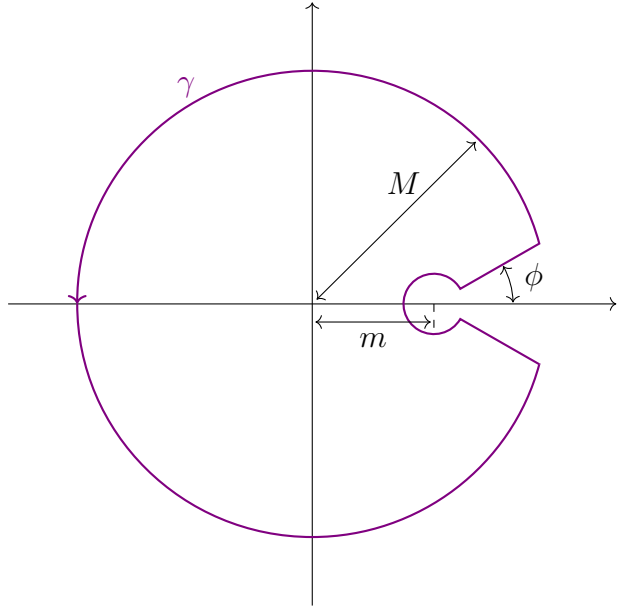
\begin{center}
\begin{tikzpicture}[scale=0.8]
\draw [->] (-5,0) -- (5,0);
\draw [->] (0,-5) -- (0,5);
\draw [thick,violet,shift={(2,0)},->]  (-30:2) -- (-30:0.5) arc (330:30:0.5) -- (30:2) arc (15:180:3.864);
\draw [thick,violet,shift={(2,0)}] (180:5.864) arc (180:345:3.864) -- (-30:2);

\draw [<->,shift={(2,0)}] (1:1.3) arc (1:29:1.3);
\draw [shift = {(2,0)}] (15:1.7) node {$\phi$};
\draw [<->,shift={(0,-0.3)}] (0.05,0) -- (1.95,0);
\draw [dashed] (2,0) -- (2,-0.4);
\draw (1,-0.6) node {$m$};
\draw [<->] (45:0.1) -- (45:3.8);
\draw (53:2.5) node {$M$};
\draw [violet] (120:4.15) node {$\gamma$};
\end{tikzpicture}
\end{center}
The small circle is chosen of radius $\frac{m}{n}$, and the path $\gamma$ is split into the following parts:
\begin{itemize}
    \item $\gamma_1=\{z\in \C \,\,|\,\, |z|=M,\,\, |\arg(z-m)| \geq \phi \}$;
    \item $\gamma_2=\{z\in \C \,\,|\,\, |z|\leq M,\,\,|z-m| \geq \frac{m}{n},\,\, \arg(z-m) = - \phi \}$;
    \item $\gamma_3=\{z\in \C \,\,|\,\, |z-m| = \frac{m}{n},\,\, |\arg(z-m)| \geq  \phi \}$;
    \item $\gamma_4=\{z\in \C \,\,|\,\, |z|\leq M,\,\,|z-m| \geq \frac{m}{n},\,\, \arg(z-m) = \phi \}$.
\end{itemize}
We have
\begin{align*}
f_n(w)\,\E^{-L(w)} &= \frac{1}{2\I \pi} \oint_{\gamma} \frac{F(z,w)\,\E^{-L(w)}}{z^{n+1}}\DD{z} \\ 
&=\frac{\Gamma(n+\theta w)}{\Gamma(\theta w)\,n!\,m^n} + \frac{1}{2\I\pi}\left(\int_{\gamma_1}+\cdots +\int_{\gamma_4}\right)\left(\frac{\E^{O(|z-m|)}-1}{(1-\frac{z}{m})^{\theta w}\,z^{n+1}}\DD{z}\right)
\end{align*}
since the $n$-th coefficient of $(1-\frac{z}{m})^{-\theta w}$ as a power series in $z$ is $\frac{\Gamma(n+\theta w)}{\Gamma(\theta w)\,n!\,m^n}$.% equivalent to n^{\theta w-1}/m^n as n goes to infinity.

\begin{lemma}\label{lem:hankel_estimate_A}
Suppose $n \geq 2$, and large enough so that $$n^2\,\left(\frac{m}{M}\right)^{n} \leq \left(\min\left(\frac{M}{m}-1,\frac{1}{\frac{M}{m}+1}\right)\right)^{\!\theta \rho}\,\E^{-Tm}.$$ Then, under the hypotheses \ref{item:hankel_1} and \ref{item:hankel_2}, 
$$\left|f_n(w)\,\E^{-L(w)}-\frac{\Gamma(n+\theta w)}{\Gamma(\theta w)\,n!\,m^n}\right|\leq A\,m^{-n}\,n^{\theta x-2}\,\E^{TM},$$
with $A=1+\frac{16}{\E}+\frac{\Gamma(2+\theta\rho)}{\E\pi (\cos \phi)^{2+\theta\rho}}$ and $x=\Re(w)$.
\end{lemma}

\begin{proof}
We split the remainder in four parts $r_{n,1}(w)+r_{n,2}(w)+r_{n,3}(w)+r_{n,4}(w)$, according to the partition $\gamma = \gamma_1 \sqcup \gamma_2 \sqcup \gamma_3 \sqcup \gamma_4$ of the Hankel contour. Suppose first $x \geq 0$. Then, it is immediate that
\begin{align*}
|r_{n,1}(w)| &\leq M^{-n}\,\left(\frac{1}{\frac{M}{m}-1}\right)^{\!\theta x} \,\E^{T(m+M)} \leq m^{-n}\,n^{\theta x-2}\,\E^{TM}
\end{align*} 
by assumption on $n$.
We then can evaluate $r_{n,2}$ and $r_{n,4}$ as follows:
\begin{align*}
|r_{n,2}(w)| &\leq \frac{T\,\E^{T(M-m)}}{2\pi} \int_{s=\frac{m}{n}}^{\infty} \left|\frac{1}{1-\frac{m+s\E^{\I\theta}}{m}}\right|^{\,\theta x}\, |m+s\E^{\I \phi}|^{-n-1}\,s\DD{s} \\ 
&\leq \frac{T\,\E^{T(M-m)}}{2\pi\,m^{n-1}}\,n^{\theta x-2} \int_{t=1}^{\infty} t^{1-\theta x}\, \left|1+\frac{t\cos \phi}{n}\right|^{-n} \DD{t} \\ 
&\leq \frac{Tm\,\E^{T(M-m)}}{2\pi}\,m^{-n}\,n^{\theta x-2} \int_{t=1}^{\infty} t^{1-\theta x} \,\E^{-t \cos \phi} \DD{t}\\ 
&\leq \frac{1}{2\pi\E\,(\cos \phi)^2}\,m^{-n}\,n^{\theta x-2}\,\E^{TM}.
\end{align*}
Finally, we have $|\exp(O(|z-m|))-1)| \leq T\exp(\frac{Tm}{n})\,\frac{m}{n}$ on $\gamma_3$, so
$$|r_{n,3}(w)| \leq 8\,Tm\exp(\frac{Tm}{n})\,m^{-n}\,n^{\theta x-2}\leq \frac{16}{\E}\,m^{-n}\,n^{\theta x-2}\,\E^{TM}.$$
Suppose now that $x<0$. Then, we get similarly
\begin{align*}
|r_{n,1}(w)| &\leq M^{-n}\,\left(\frac{1}{\frac{M}{m}+1}\right)^{\!\theta x} \,\E^{T(m+M)} ;\\ 
|r_{n,2}(w)| &\leq  \frac{Tm\,\E^{T(M-m)}\,\Gamma(2-\theta x)}{2\pi\,(\cos \phi)^{2-\theta x }}\,m^{-n}\,n^{\theta x-2};\\ 
|r_{n,3}(w)| &\leq 8\,Tm\,\exp(\frac{Tm}{n})\,m^{-n}\,n^{\theta x-2},
\end{align*}
whence the result in both cases.
\end{proof}

\begin{lemma}\label{lem:hankel_estimate_B}
If $n\geq  2\theta \rho + 1$ and $w \in D(0,\rho)$, then
$$\left|\frac{\Gamma(n+\theta w)}{n!} - n^{\theta w -1}\right| \leq B\,n^{\theta x-2}$$
with $B=3\left(\theta\rho+\frac{1}{2}\right)^2\,\E^{\frac{3}{2}(\theta\rho+\frac{1}{2})}$ and $x=\Re(w)$.
\end{lemma}
\begin{proof}
The easiest proof consists in using an explicit form of Stirling estimates, namely,
$$\log \Gamma(z+1) = \left(z+\frac{1}{2}\right)\log(z+\frac{1}{2})-\left(z+\frac{1}{2}\right) + \frac{1}{2}\log(2\pi) + \sum_{k=1}^\infty \int_0^{\frac{1}{2}} \log(1-\frac{t^2}{(z+k)^2})\DD{t}.$$
This formula is valid for any complex number $z$ such that $\Re(z)>0$. Therefore,
\begin{align*}
&\log \Gamma(n+\theta w) - \log(n!) - \left(\theta w -1\right)\log n \\ 
&=\left(n+\theta w -\frac{1}{2}\right)\log(1 +\frac{\theta w-\frac{1}{2}}{n}) - \left(n+\frac{1}{2}\right)\log(1+\frac{1}{2n}) -\left(\theta w -1\right)  \\ 
&\quad  + \sum_{k=1}^\infty \int_0^{\frac{1}{2}} \log(1-\frac{t^2}{(n+\theta w-1+k)^2}) - \log(1-\frac{t^2}{(n+k)^2}) \DD{t}.
\end{align*}
By using the integral Taylor formula $\log(1+z)-z = - z^2 \int_0^1 \frac{1-t}{(1+tz)^2}\DD{t}$, one shows that the first line is smaller than 
$$(6 \log 2 -2)\,\frac{(\theta \rho + \frac{1}{2})^2}{n}.$$ 
On the second line, the function $\log(1+z)$ is Lipschitz with constant $\frac{4}{3}$ for arguments with modulus smaller than $\frac{1}{4}$, so 
\begin{align*}
\left|\log(1-\frac{t^2}{(n+\theta w-1+k)^2}) - \log(1-\frac{t^2}{(n+k)^2})\right|&\leq \frac{4t^2}{3}\left|\frac{1}{(n+\theta w-1+k)^2}-\frac{1}{(n+k)^2}\right| \\
&\leq 8t^2\,\frac{\theta\rho+1}{(n+k)^3} \leq \frac{4t^2}{(n+k)^2},
\end{align*}
and the series is smaller than $\frac{1}{6}\sum_{k=1}^\infty \frac{1}{(n+k)^2} \leq \frac{1}{6n}$. Putting everything together and simplifying a bit the expression, we obtain:
$$|\log \Gamma(n+\theta w) - \log(n!) - \left(\theta w -1\right)\log n| \leq 3\,\frac{(\theta \rho + \frac{1}{2})^2}{n}.$$
The estimate follows readily by taking the exponential.
\end{proof}

\begin{theorem}\label{thm:hankel_estimate}
We place ourselves under the hypotheses \ref{item:hankel_1} and \ref{item:hankel_2}, and we also assume that $n$ is large enough, so that the conditions of Lemmas \ref{lem:hankel_estimate_A} and \ref{lem:hankel_estimate_B} hold. Then, the rescaled moment generating series $\esper[w^{X_n}]$ satisfies the following uniform estimate over $D(0,\rho)$:
$$\esper[w^{X_n}]\,\E^{-\theta (\log n)(w-1)-(L(w)-L(1))}= \left(\frac{\Gamma(\theta)}{\Gamma(\theta w)} + O\left(\frac{A'+B\left|\frac{\Gamma(\theta)}{\Gamma(\theta w)}\right|}{n}\right)\right) \left(1+O\left(\frac{A'+B}{n}\right)\right)^{-1},$$
where \begin{align*}
A'&=\left(1+\frac{16}{\E}+\frac{\Gamma(2+\theta\rho)}{\E\pi (\cos \phi)^{2+\theta\rho}}\right)\Gamma(\theta)\,\E^{TM}\qquad ;\qquad 
B= 3\left(\theta\rho+\frac{1}{2}\right)^2\,\E^{\frac{3}{2}(\theta\rho+\frac{1}{2})},
\end{align*}
and where the implied constants in the $O(\cdot)$'s are both equal to $1$.
\end{theorem}
\begin{proof}
This follows immediately from Lemmas \ref{lem:hankel_estimate_A} and \ref{lem:hankel_estimate_B} and from the formula
$\esper[w^{X_n}] = \frac{f_n(w)}{a_n} = \frac{f_n(w)}{f_n(1)} $.
\end{proof}
\medskip

Let us explain how to use Theorem \ref{thm:hankel_estimate} in order to prove the assumption of Theorem \ref{main:C} for a sequence of random variables $(X_n)_{n \in \N}$. Since $L$ is holomorphic on $D(0,\rho)$ with $\rho>1$, we have a convergent power series
$$L(w)-L(1) = K(w-1) + \sum_{k=2}^\infty \frac{(-1)^{k-1}\,p_k}{k}\,(w-1)^k$$
for some coefficients $K$ and $p_{k \geq 2}$. For the sequence $(C_n)_{n \in \N}$ of numbers of cycles, $K$ is given by the assumption on $\mathfrak{P}(\Theta,z)$ and $p_{k \geq 2} = 0$; whereas for the sequence $(D_n)_{n \in \N}$ of numbers of irreducible divisors, $K=R_q=\mathfrak{R}_q(q^{-1})$ and $p_k=\mathfrak{I}_q(q^{-k})$. Notice on the other hand that the infinite product representation of the $\Gamma$ function yields:
\begin{align*}
\frac{\Gamma(\theta)}{\Gamma(\theta w)} &= \prod_{n=1}^\infty \left(1+\frac{\theta}{n+\theta-1}\,(w-1)\right) \left(1+\frac{1}{n}\right)^{\!-\theta (w -1)} \\ 
&=\frakE(A_\theta',w-1) \,\E^{\gamma_\theta (w-1)}\quad\text{with }\begin{cases}
    A_\theta = \{\frac{\theta}{\theta+n-1},\,\,n \geq 1\}, &\\
    \,\gamma_\theta = \sum_{n=1}^\infty \frac{\theta}{n+\theta-1}-\theta \log(1+\frac{1}{n}).&
\end{cases}
\end{align*}
Therefore, if the parameters $p_{k \geq 2}$ can be written as a specialisation $\frakp_k(B)$ with $B$ square-summable family, then Theorem \ref{thm:hankel_estimate} yields a mod-Poisson convergence result with Bernoulli asymptotics. The parameters are $$\lambda_n=\theta \log n + K + \gamma_\theta,$$ the limiting alphabet is $$A = A_\theta \sqcup B,$$ and the remainder $\eps_n$ is a $O(n^{-1})$ with an explicit constant. In particular, $\eps_n \ll (\lambda_n)^{-\frac{r+3}{2}}$ for any $r \geq 1$, so $\dtv(\mu_n,\nu_{n,*}^{(r)}) = O((\lambda_n)^{-\frac{r+1}{2}})$ for any order of approximation $r \geq 1$. Our discussion yields a proof of Propositions \ref{prop:random_permutations} and \ref{prop:random_polynomials}, and it establishes the validity of our approach for these models.
\medskip

\begin{remark}\label{rem:next}
The analysis of the sequence $(\omega_n)_{n \geq 1}$ of numbers of prime divisors follows the same ideas as above, except that the double generating series $F(z,w)$ needs to be replaced by the bivariate $L$-series
$$F(s,w) = \sum_{n=1}^\infty \frac{w^{\omega(n)}}{n^s} = \prod_{p \in \mathbb{P}} \left(1+\frac{w}{p^s-1}\right).$$
This series is convergent for any $s=\sigma+\I \tau$ with $\sigma>1$; in the sequel we set $\kappa=1+\frac{1}{\log x}$. By the Perron integral formula, if $x$ is not an integer, then
$$A(x,w)=\sum_{n\leq x} w^{\omega(n)} = \frac{1}{2\I\pi}\int_{\kappa-\I \infty}^{\kappa+\I \infty} F(s,w)\,\frac{x^s}{s}\,\DD{s};$$
see \cite[Chapter II.2]{Ten95}. On the other-hand, if $x=n$ is an integer, then $A(n,w) = n\,\esper[w^{\omega_n}]$. The Selberg--Delange method consists in:
\begin{itemize}
    \item writing $F(s,w)$ as a product $G(s,w)\,(\zeta(s))^w$, where $w$ is arbitrary in $D(0,\rho)$  with $\rho >1$, and $G(s,w)$ is a biholomorphic function on the domain $\{s=\sigma+\I \tau\,|\,\sigma > \frac{1}{2}\} \times D(0,\rho)$. 
    \item estimating the integral $A(x,w)$ by deforming the contour $\kappa -\I \infty \to \kappa + \I \infty$ into a Hankel contour as below:
    
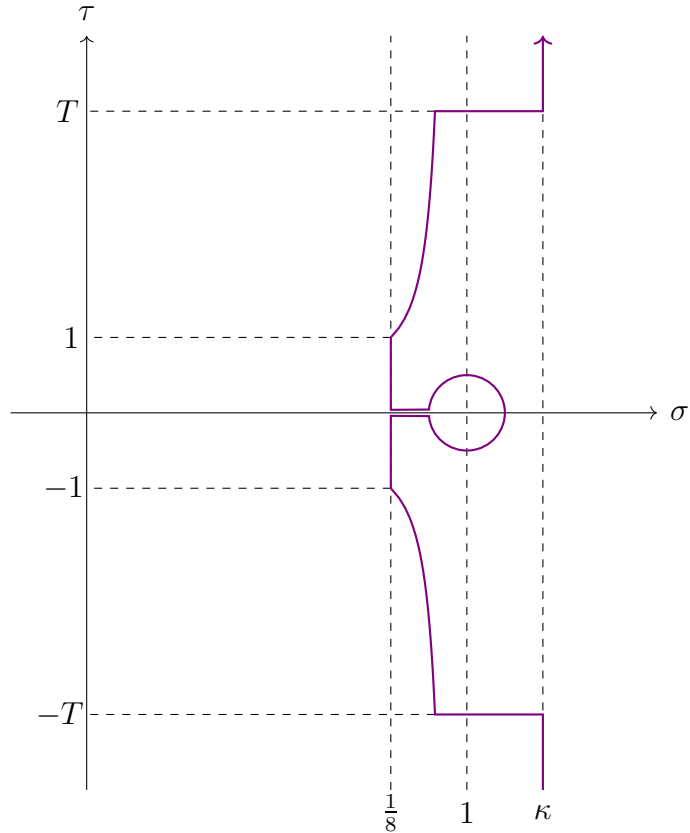
\begin{center}
\begin{tikzpicture}[scale=1]
\draw [->] (-1,0) -- (7.5,0);
\draw [->] (0,-5) -- (0,5);
\draw (7.8,0) node {$\sigma$};
\draw (0,5.3) node {$\tau$};
\draw (4,-5.3) node {$\frac{1}{8}$};
\draw (5,-5.3) node {$1$};
\draw (6,-5.3) node {$\kappa$};
\draw [dashed] (5,-5) -- (5,5);
\draw [dashed] (6,-5) -- (6,5);
\draw [dashed] (4,-5) -- (4,5);
\draw [dashed] (4,1) -- (0,1);
\draw [dashed] (4,-1) -- (0,-1);
\draw [dashed] (4.581,4) -- (0,4);
\draw [dashed] (4.581,-4) -- (0,-4);
\draw (-0.2,1) node {$1$};
\draw (-0.25,4) node {$T$};
\draw (-0.35,-4) node {$-T$};
\draw (-0.3,-1) node {$-1$};
\draw [thick,violet,shift={(5,0)}] (-1,-1) -- (-1,-0.04) -- (185:0.5) arc (-175:175:0.5) -- (-1,0.04) -- (-1,1);
\draw [thick,violet,domain=1:4,->] plot ({5-1/(1+ln(\x))},{\x}) -- (6,4) -- (6,5);
\draw [thick,violet,domain=1:4] plot ({5-1/(1+ln(\x))},{-\x}) -- (6,-4) -- (6,-5);
\end{tikzpicture}
\end{center}
The form of this contour is related to the existence of a zero-free region for the Riemann $\zeta$ function. It is the piecewise smooth path 
$$\gamma = \gamma_4^- \sqcup \gamma_3^- \sqcup \gamma_2^- \sqcup \gamma_1^- \sqcup \gamma_0 \sqcup \gamma_1^+ \sqcup \gamma_2^+ \sqcup \gamma_3^+ \sqcup \gamma_4^+ $$
with the following parts:
\begin{itemize}
    \item $\gamma_0$ is the union of the circle with center $1$ and radius $\frac{1}{2\log x}$, and of the two horizontal lines $\frac{1}{8} + \I 0_- \to 1-\frac{1}{2\log x} + \I 0_-$ and $1-\frac{1}{2\log x} + \I 0_+ \to \frac{7}{8} + \I 0_+$.
    \item $\gamma_1^\pm$ is the vertical line which connects $\frac{7}{8} +\I 0_{\pm}$ to $\frac{7}{8} \pm \I$.
    \item $\gamma_2^\pm$ is the path $1-\frac{1}{8(1+\log t)} \pm\I t$ for $t \in [1,T]$.
    \item $\gamma_3^\pm$ is the horizontal line which connects $1-\frac{1}{8(1+\log T)} \pm\I T$ to $\kappa\pm \I T$. 
    \item $\gamma_4^\pm$ is the vertical line which connects $\kappa \pm \I T$ to $\kappa \pm \I \infty$.
\end{itemize}
\end{itemize}
The parameter $T$ is a large real number which can be chosen optimally according to the value of $x$. One can show that if $\rho=\frac{5}{4}$ and $s$ is on the right of the contour $\gamma$, then:
\begin{enumerate}[label=(Z\arabic*)]
    \item\label{item:zeta_properties_1} for any non-trivial root $z$ of the Riemann $\zeta$ (hence, in the critical strip $0<\Re (z) < 1$), $\frac{1}{z}+\frac{1}{s-z}$ has positive real part.
    \item\label{item:zeta_properties_2} $|\log \zeta(s) | \leq \log(1+\log_+ |\Im(s)|) + C$ for some constant $C$.
    \item $G(s,w)$ is uniformly bounded.
\end{enumerate}
Following closely the argument from \cite[Chapter II.5]{Ten95}, one can prove that 
$$\eps_n = \sup_{w \in D(0,\frac{5}{4})} \left|\esper[w^{\omega_n}]\,\E^{-(\log \log n+\gamma)(w-1)} - \frakE(A_\omega',w-1) \right| \leq \frac{K}{\log n}$$
for some explicit constant $K$. Note that all this computation of the remainder $\eps_n$ is unconditional and does not require the Riemann hypothesis (the RH would only enable one to get a much better constant $K$). We plan to explain how to compute a good constant $K$ in a forthcoming work, by modifying a bit the Selberg--Delange method. The compensation argument which consists in introducing $G(s,w)=F(s,w)\,(\zeta(s))^{-w}$ and in using known estimates of the complex powers of $\zeta$ in order to control $$\frac{1}{2\I \pi} \int_\gamma G(s,w)\,(\zeta(s))^w\,\frac{x^s}{s}\DD{s}$$ allows one to deal with general (bivariate) $L$-series $F(s,w)$, but it leads to a loss of accuracy if $F(s,w)$ is explicitly known. Hence, in order to study the sequence $(\omega_n)_{n \in \N}$, one can work directly with the series $F(s,w)$ and a zero-free region of its meromorphic extension, and prove the analogues of \ref{item:zeta_properties_1} and \ref{item:zeta_properties_2} directly for $F(s,w)$. We shall show these computations from analytic number theory in a separate work.
\end{remark}
\bigskip
\bigskip

\printbibliography

\end{document}